\documentclass[11pt]{amsart}
\usepackage[T1]{fontenc}
\usepackage[utf8]{inputenc}
\usepackage{lmodern,amsmath,amssymb,mathtools,microtype}
\usepackage[margin=1in]{geometry}
\usepackage{tikz}
\usetikzlibrary{arrows.meta,calc,positioning,patterns,decorations.pathreplacing}
\definecolor{cvblue}{RGB}{26,87,128}
\definecolor{cvorange}{RGB}{173,87,21}
\definecolor{cvred}{RGB}{163,45,45}
\tikzset{
 cv circle/.style={draw=black!30,line width=.5pt},
 cv orbit/.style={-{Stealth[length=2mm]},draw=cvblue,line width=.85pt,shorten <=2.5pt,shorten >=3pt},
 cv second/.style={cv orbit,draw=cvorange,dashed},
 cv transient/.style={-{Stealth[length=1.7mm]},draw=black!42,line width=.55pt,shorten <=2pt,shorten >=3pt},
 cv axis/.style={-{Stealth[length=1.7mm]},line width=.5pt},
 cv every/.style={font=\small},
}

\usepackage[colorlinks=true,linkcolor=blue!55!black,citecolor=blue!55!black,urlcolor=blue!55!black]{hyperref}
\hypersetup{pdftitle={Homotopy distributions for Vietoris--Rips complexes of random circle samples},pdfauthor={Jaemin Park}}
\numberwithin{equation}{section}
\newtheorem{theorem}{Theorem}[section]
\newtheorem{proposition}[theorem]{Proposition}
\newtheorem{lemma}[theorem]{Lemma}

\theoremstyle{definition}

\theoremstyle{remark}
\newtheorem{remark}[theorem]{Remark}
\DeclareMathOperator{\VR}{VR}
\DeclareMathOperator{\Leb}{Leb}
\DeclareMathOperator{\Pois}{Poisson}
\DeclareMathOperator{\pr}{pr}
\newcommand{\R}{\mathbb R}
\newcommand{\Z}{\mathbb Z}

\newcommand{\Pp}{\mathbb P}
\newcommand{\E}{\mathbb E}
\newcommand{\ind}{\mathbf 1}
\newcommand{\pt}{\mathrm{pt}}
\newcommand{\Qn}{\mathcal Q_n}
\newcommand{\Xn}{\mathcal X_n}
\title[Homotopy distributions of random circle samples]{Homotopy distributions for Vietoris--Rips complexes of random circle samples}
\author{Jaemin Park}
\address{Department of Mathematics and Research Institute of Natural Sciences, Sookmyung Women's University, Cheongpa-ro 47-gil 100, Yongsan-gu, Seoul, 04310, Korea}
\date{October 16, 2026}
\AddToHook{cmd/selectfont/after}{%
  \baselineskip=\dimexpr\baselineskip\relax plus 1.2pt minus .1pt\relax}
\AddToHook{cmd/normalsize/after}{%
  \abovedisplayskip=10pt plus 1pt minus 3pt
  \belowdisplayskip=10pt plus 1pt minus 3pt
  \abovedisplayshortskip=0pt plus 1pt
  \belowdisplayshortskip=6pt plus 1pt minus 3pt}
\begin{document}
\begin{abstract}
We determine the distribution of the homotopy type of the Vietoris--Rips complex of a fixed number of independent uniform points on a circle. An exact identity expresses each binomial moment of a periodic-orbit count as one probability for circular gaps, multiplied by a constant avoidance factor. We evaluate these probabilities by finite lattice-point counts and interpolation, obtaining explicit formulas for every sample size and every real scale. On an explicit rational partition, the probabilities are continuous piecewise polynomials of degree at most one less than the sample size. We also determine the homotopy distribution in the first transition window above each positive critical fraction $c_t=t/(2t+1)$. In this window the number of critical periodic orbits converges to a Poisson law with explicit intensity, yielding the limiting probabilities of the preceding odd sphere, the contractible type, and every wedge multiplicity of the critical even sphere.
\end{abstract}
\maketitle

\section{Introduction}
The Vietoris--Rips complex of a finite subset of a circle need not have the homotopy type of a circle. Depending on the sample and the scale, its homotopy type can be a discrete space, a point, an odd-dimensional sphere, or a wedge of even-dimensional spheres. The deterministic classification is due to Adamaszek and Adams~\cite{AA} and its formulation through periodic orbits to Adamaszek, Adams and Motta~\cite{AAM}. We study the probability of each type for a uniform random sample.

Let $S^1=\R/\Z$ be the circle of circumference one, equipped with normalized length measure $\mu$ and the geodesic metric. Write $\pi:\R\to S^1$ for the quotient map. For each integer $n\geq1$, use the probability space
\begin{equation}\label{eq:sample-space}
 (\Omega_n,\mathcal F_n,\Pp_n)
 =\bigl([0,1)^n,\mathcal B([0,1))^{\otimes n},
              (\Leb_1|_{[0,1)})^{\otimes n}\bigr),
 \qquad X_i(\omega)=\pi(\omega_i)\quad(1\leq i\leq n),
\end{equation}
and let $\Xn(\omega)=\{X_1(\omega),\ldots,X_n(\omega)\}$. Expectations under $\Pp_n$ are denoted by $\E_n$. We use the closed Vietoris--Rips convention: a finite subset is a simplex of $\VR(\Xn;r)$ if all its pairwise distances are at most $r$. For a homotopy type $T$, put
\begin{equation}\label{eq:target}
 p_T(n,r)=\Pp_n\{\omega\in\Omega_n:\VR(\Xn(\omega);r)\simeq T\}.
\end{equation}
Our two main results give the full finite-sample distribution and its limit in a critical window.

The endpoint cases are elementary. A one-point sample is contractible at every scale. For $n\geq2$, the complex at $r=0$ is almost surely $n$ discrete points, and it is a simplex at $r\geq1/2$. Henceforth the nontrivial scale range is $0<r<1/2$. We write $\pt$ for the contractible type, interpret $\bigvee^m S^0$ as $m+1$ points, and interpret the empty wedge as a point.

\subsection{The finite-sample homotopy law}
The formulas use the finite set of possible reduced winding fractions
\begin{equation}\label{eq:q-set}
 \Qn=\{0\}\cup
 \left\{\frac w\ell: w,\ell\in\Z,\ 1\leq\ell\leq n,\
           0<w<\frac\ell2,\ \gcd(w,\ell)=1\right\}.
\end{equation}
We first specify the coefficients by finite arithmetic; their probabilistic meaning is proved in Sections~\ref{sec:moments} and~\ref{sec:evaluation}.

For integers $M\geq3$, $1\leq s<M/2$, and $K,H\in\Z_{\geq0}$, define
\begin{equation}\label{eq:count}
 C_{M,s}(K,H)
 =\#\left\{d\in\Z_{\geq0}^{M}:
 \sum_{i=0}^{M-1}d_i=K,\quad
 \sum_{v=0}^{s-1}d_{i+v}\leq H\leq\sum_{v=0}^{s}d_{i+v}
 \quad(i=0,\ldots,M-1)\right\}.
\end{equation}
Here $\#$ denotes the number of elements in a finite set. Thus $C_{M,s}(K,H)$ counts nonnegative integer vectors with total $K$ satisfying all the displayed block inequalities. All subscripts of $d$ are interpreted modulo $M$. In particular, $C_{M,s}(0,0)=1$.

For integers $M\geq2$ and $0\leq s<M/2$, put
\begin{equation}\label{eq:support}
 I_{M,s}=\left(\frac{s}{M},\frac{s+1}{M}\right)\cap(0,1/2),
 \qquad
 \mathcal R_M=\left\{\frac{p}{v}\in[0,1/2]:
             p\in\Z_{\geq0},\ v\in\{1,\ldots,M\}\right\}.
\end{equation}
Repeated fractions in $\mathcal R_M$ are identified. For $M\geq3$ and $1\leq s<M/2$, the coefficient $P_{M,s}(r)$ is zero for $r\notin I_{M,s}$. At a rational scale $r=A/B\in I_{M,s}$, where $A,B$ are coprime positive integers, it is
\begin{equation}\label{eq:rational-coefficient}
 P_{M,s}(A/B)=\frac{1}{B^{M-1}}
 \sum_{k=0}^{M-1}(-1)^{M-1-k}\binom{M-1}{k}
                         C_{M,s}(kB,kA).
\end{equation}
For an irrational $r\in I_{M,s}$, let $\alpha=A_-/B_-<\beta=A_+/B_+$ be the consecutive elements of $\mathcal R_M$ containing $r$, written in lowest terms. Define
\begin{equation}\label{eq:gap-interpolation}
 \begin{split}
 \rho_k&=\frac{kA_-+A_+}{kB_-+B_+},\qquad 1\leq k\leq M,\\
 P_{M,s}(r)&=\sum_{k=1}^{M}P_{M,s}(\rho_k)
     \prod_{\substack{1\leq v\leq M\\v\ne k}}
                  \frac{r-\rho_v}{\rho_k-\rho_v}.
 \end{split}
\end{equation}
The numbers $\rho_1,\ldots,\rho_M$ are distinct rational points in $(\alpha,\beta)$. To evaluate $P_{M,s}(\rho_k)$, write $\rho_k$ in lowest terms and apply~\eqref{eq:rational-coefficient}. We prove that this Lagrange interpolation agrees with the rational formula throughout the interval.

For integers $1\leq M\leq n$ and $0\leq s<M/2$, define
\begin{equation}\label{eq:weights}
 W_{M,s}(n,r)=
 \begin{cases}
  \displaystyle\binom nM(1-Mr)_+^{n-1},&\text{if }s=0,\\[3pt]
  \displaystyle\binom nM(1-Mr+s)^{n-M}P_{M,s}(r),
       &\text{if }s\geq1,\ r\in I_{M,s},\\[3pt]
  0,&\text{if }s\geq1,\ r\notin I_{M,s},
 \end{cases}
\end{equation}
where $x_+=\max\{x,0\}$. The support condition is checked before the power is formed.

Although $P_{M,s}$ was defined above only for $M\geq3$, the case $s=0$ in~\eqref{eq:weights} does not involve $P_{M,s}$. Hence $W_{M,0}$ is also defined for $M=1,2$ whenever $M\leq n$. For integer $s\geq1$, the condition $s<M/2$ already implies $M\geq3$.

Proposition~\ref{prop:moments} computes the binomial moments from gap measures; Section~\ref{sec:finite-completion} evaluates them to obtain the weights in~\eqref{eq:weights}.

For the homotopy classification, write $c_t=t/(2t+1)$ for $t\in\Z_{\geq0}$.

\begin{theorem}[Homotopy distribution]\label{thm:finite}
For $n\geq2$, $0<r<1/2$, $t\in\Z_{\geq0}$ and $m\in\Z_{\geq1}$,
\begin{align}
 p_{S^{2t+1}}(n,r)
 &=\sum_{\substack{q=w/\ell\in\Qn\\c_t<q<c_{t+1}}}
       \sum_{j=1}^{\lfloor n/\ell\rfloor}
                   (-1)^{j-1}W_{j\ell,jw}(n,r),\label{eq:odd-law}\\
 p_{\bigvee^m S^{2t}}(n,r)
 &=\sum_{j=m+1}^{\lfloor n/(2t+1)\rfloor}
       (-1)^{j-m-1}\binom{j}{m+1}
                   W_{j(2t+1),jt}(n,r),\label{eq:even-law}\\
 p_{\pt}(n,r)
 &=\sum_{t=0}^{\lfloor(n-1)/2\rfloor}
       \sum_{j=1}^{\lfloor n/(2t+1)\rfloor}
                   (-1)^{j-1}j\,W_{j(2t+1),jt}(n,r).\label{eq:point-law}
\end{align}
All other homotopy types have probability zero. Each $p_T(n,\cdot)$ is continuous on $(0,1/2)$ and polynomial of degree at most $n-1$, with rational coefficients, on every interval between consecutive elements of $\mathcal R_n$.
\end{theorem}

The reason for treating $\pt$ separately from the even-dimensional wedges is explained in the proof in Section~\ref{sec:finite-completion}. Thus the result specifies the probabilities at rational and irrational scales without an unevaluated integral. The partition consists of possible changes of the probability polynomials; it does not describe the changes of a fixed sample's complex. The same weight is reused whenever its pair $(M,s)$ occurs. There are only $n+\sum_{M=2}^n\lfloor(M-1)/2\rfloor$ such pairs. This is a bound on the number of coefficient families, not a claim of polynomial-time evaluation.

\subsection{The limiting homotopy law in a critical window}
Fix $t\in\Z_{\geq1}$ and $a>0$, and set
\begin{equation}\label{eq:critical-parameters}
 \ell=2t+1,\qquad
 r_n=c_t+a n^{-\ell/(\ell-1)},\qquad
 \lambda=\frac{\ell^{\ell-1}a^{\ell-1}}{\ell!}.
\end{equation}
For all sufficiently large $n$, $r_n<1/2$.

\begin{theorem}[Limiting homotopy distribution]\label{thm:critical}
For the sequence in~\eqref{eq:critical-parameters} and every integer $m\geq1$,
\begin{align}
 \lim_{n\to\infty}p_{S^{2t-1}}(n,r_n)&=e^{-\lambda},\label{eq:critical-odd}\\
 \lim_{n\to\infty}p_{\pt}(n,r_n)&=\lambda e^{-\lambda},\label{eq:critical-point}\\
 \lim_{n\to\infty}p_{\bigvee^m S^{2t}}(n,r_n)
       &=e^{-\lambda}\frac{\lambda^{m+1}}{(m+1)!}.\label{eq:critical-even}
\end{align}
The total probability of homotopy types outside this family tends to zero as $n\to\infty$.
\end{theorem}

The random variable behind this law is the number of periodic orbits with winding fraction $c_t$. Its distribution converges to $\Pois(\lambda)$. With probability tending to one, zero critical orbits give $S^{2t-1}$, one gives a point, and $m+1$ give $\bigvee^m S^{2t}$.

\subsection{Proof outline and organization of the paper}
We use the cyclic-graph and periodic-orbit framework of Adamaszek and Adams~\cite{AA} and Adamaszek, Adams and Motta~\cite{AAM}. Binomial moments of the orbit counts are computed by conditioning on the circular gaps of selected orbit unions. The resulting gap-region probabilities are evaluated by lattice-point counts and polynomial interpolation, giving the finite-sample law. In the critical window, the moment estimates yield a Poisson limit for the orbit count, and winding localization converts this into the limiting homotopy law.

Section~\ref{sec:model} introduces the cyclic dynamics and probability inversion. Section~\ref{sec:moments} constructs the gap maps and proves the all-order orbit-moment identity in terms of gap measures. Section~\ref{sec:evaluation} studies their real-parameter geometry and evaluates them by lattice counts and interpolation to prove Theorem~\ref{thm:finite}. Section~\ref{sec:critical} proves the critical one-orbit formula, controls interactions between candidate orbits, and proves Theorem~\ref{thm:critical}.

\section{Sampling, cyclic dynamics and homotopy types}\label{sec:model}
For $x,y\in S^1$, let $\vec d(x,y)\in[0,1)$ denote clockwise separation. The geodesic distance is
\[
 d(x,y)=\min\{\vec d(x,y),\vec d(y,x)\}.
\]
We identify the increasing coordinate on $\R/\Z$ with the clockwise orientation. All random objects below are measurable on the product spaces~\eqref{eq:sample-space}. Coincidences of sampled points form a null event. We may assign arbitrary fixed values to orbit counts on this event without changing their laws.

\subsection{The farthest-clockwise map}
Let $X\subset S^1$ consist of $n$ distinct points and fix $0<r<1/2$. Write $\widetilde X=\pi^{-1}(X)$ and choose its increasing enumeration $(x_i)_{i\in\Z}$ with $x_{i+n}=x_i+1$. Define $F_r:\widetilde X\longrightarrow\widetilde X$ by
\begin{equation}\label{eq:farthest}
 F_r(x_i)=x_{\max\{b\in\Z:b\geq i,\ x_b-x_i<r\}},
 \qquad i\in\Z.
\end{equation}
The maximum exists because $b=i$ is eligible and $\lim_{b\to+\infty}x_b=+\infty$. The relation $x_{i+n}=x_i+1$ gives
\[
 F_r(x_i+k)=F_r(x_i)+k,\qquad i,k\in\Z.
\]
We therefore define $f_r:X\longrightarrow X$ to be the unique map such that
\[
 f_r\circ\pi=\pi\circ F_r\quad\text{on }\widetilde X.
\]

Thus $f_r$ sends a point to the clockwise-farthest sampled point at distance strictly less than $r$. The maximal eligible index is nondecreasing as $i$ increases. The resulting circle map preserves cyclic order, as in~\cite[Definition~2.1]{AAM}.

For each fixed $n$ and deterministic $r$, the event that two distinct sampled points have clockwise separation exactly $r$ is $\Pp_n$-null. Away from this event the strict and closed distance conventions coincide. The graph joining two points whenever one of their clockwise separations is less than $r$ has clique complex $\VR_{<}(X;r)$. We may therefore use the strict dynamics to compute the closed-complex probabilities~\eqref{eq:target}. All such null-event assertions are for a fixed scale, not simultaneously for every real scale.

A periodic orbit has a least period $\ell$, measured in iterations, and winding number $w\in\Z_{\geq0}$ specified by
\begin{equation}\label{eq:period-winding}
 F_r^{\ell}(\widetilde x)=\widetilde x+w.
\end{equation}
The period is the number of points on the orbit, not its geodesic length. The winding fraction $q=w/\ell$ is reduced, and all periodic orbits of a finite cyclic map share this same pair $(w,\ell)$~\cite[Lemma~2.3 and preceding discussion]{AAM}. Each step is shorter than $r$, so $q<r<1/2$ and $q\in\Qn$.

For $q=w/\ell\in\Qn$, written in lowest terms with $0$ represented as $0/1$, define the random variable $N_q(n,r):\Omega_n\to\Z_{\geq0}$ to be the number of periodic orbits of fraction $q$; it is zero when the sample's winding fraction differs from $q$. Exactly one of these counts is positive on the distinct-point locus, and
\begin{equation}\label{eq:orbit-bound}
 0\leq N_q(n,r)\leq\lfloor n/\ell\rfloor.
\end{equation}
Measurability follows since the finite map is specified by finitely many Borel comparisons of sample coordinates.

\begin{samepage}
The homotopy classification can be stated in terms of the common winding fraction and the number of periodic orbits~\cite[Proposition~6.4]{AAM}.
\begin{proposition}[Adamaszek--Adams--Motta, 2017]\label{prop:classifier}
Let $X$ be a finite set of distinct circle points, $0<r<1/2$, and let $q$ and $N$ be the common winding fraction and the number of periodic orbits of $f_r$. Then
\[
 \VR_{<}(X;r)\simeq
 \begin{cases}
 S^{2t+1},&\text{if }c_t<q<c_{t+1},\\
 \displaystyle\bigvee^{N-1}S^{2t},&\text{if }q=c_t,
 \end{cases}
 \qquad t\in\Z_{\geq0}.
\]
The same classification holds $\Pp_n$-almost surely for $\VR(\Xn;r)$ at each fixed $r$.
\end{proposition}
\end{samepage}

Figure~\ref{fig:dynamics} shows why the orbit count is essential even when the winding fraction is fixed. In the complex diagrams, every clique of the displayed graph is filled. Following the closed-neighborhood argument in~\cite[Lemma~4.1, proof]{AA}, the gray vertices can be removed successively: each has its closed graph neighborhood contained in that of a retained vertex. Sending the removed vertex to that retained vertex gives a simplicial retraction whose composite with inclusion is homotopic to the identity, so these deletions preserve the homotopy type.

\begin{figure}[tb]
\centering
\begin{tikzpicture}[cv every]
 \node at (1.5,4.15) {farthest-point map};
 \node at (6.0,4.15) {Vietoris--Rips complex};
 \node at (11.0,4.15) {after vertex deletions};
 \begin{scope}[shift={(1.5,2.4)}]
  \draw[cv circle] (0,0) circle (1.2);
  \foreach \i/\x in {0/0,1/.08,2/.34,3/.46,4/.67,5/.84}
   \coordinate (a\i) at ({90-360*\x}:1.2);
  \foreach \i/\j in {1/2,3/4,5/1} \draw[cv transient] (a\i)--(a\j);
  \foreach \i/\j in {0/2,2/4,4/0} \draw[cv orbit] (a\i)--(a\j);
  \foreach \i in {1,3,5} \fill[black!48] (a\i) circle (1.8pt);
  \foreach \i in {0,2,4} \fill[cvblue] ($(a\i)+(-2pt,-2pt)$) rectangle ($(a\i)+(2pt,2pt)$);
  \node at (0,-1.6) {\textup{(a)} $N_{1/3}(n,r)=1$};
 \end{scope}
 \draw[-{Stealth[length=2mm]}] (3.0,2.4)--(4.1,2.4);
 \begin{scope}[shift={(6.0,2.4)}]
  \draw[cv circle] (0,0) circle (1.2);
  \foreach \i/\x in {0/0,1/.08,2/.34,3/.46,4/.67,5/.84}
   \coordinate (v\i) at ({90-360*\x}:1.2);
  \foreach \i/\j/\k in {0/1/2,0/1/5,0/2/4,0/4/5,2/3/4}
   \fill[cvblue!12] (v\i)--(v\j)--(v\k)--cycle;
  \fill[cvblue!25] (v0)--(v2)--(v4)--cycle;
  \foreach \i/\j in {0/1,0/2,0/4,0/5,1/2,1/5,2/3,2/4,3/4,4/5}
   \draw[draw=cvblue!65,line width=.55pt] (v\i)--(v\j);
  \foreach \i in {1,3,5} \fill[black!48] (v\i) circle (2pt);
  \foreach \i in {0,2,4} \fill[cvblue] ($(v\i)+(-2pt,-2pt)$) rectangle ($(v\i)+(2pt,2pt)$);
  \node at (0,-1.6) {filled triangles};
 \end{scope}
 \draw[-{Stealth[length=2mm]}] (7.5,2.4)--(8.8,2.4);
 \node[font=\footnotesize,align=center] at (8.15,1.7) {delete gray\\vertices};
 \begin{scope}[shift={(10.2,2.4)}]
  \coordinate (c0) at (0,.9);
  \coordinate (c2) at (.85,-.6);
  \coordinate (c4) at (-.85,-.6);
  \filldraw[fill=cvblue!25,draw=cvblue,line width=.7pt] (c0)--(c2)--(c4)--cycle;
  \foreach \i in {0,2,4} \fill[cvblue] ($(c\i)+(-2pt,-2pt)$) rectangle ($(c\i)+(2pt,2pt)$);
 \end{scope}
 \draw[-{Stealth[length=2mm]}] (11.4,2.4)--(12.35,2.4);
 \node[above=2pt] at (11.87,2.4) {$\simeq$};
 \fill[cvblue] (12.7,2.4) circle (2.5pt);
 \node[below=5pt] at (12.7,2.4) {$\pt$};

 \begin{scope}[shift={(1.5,-1.65)}]
  \draw[cv circle] (0,0) circle (1.2);
  \foreach \i/\x in {0/0,1/.05,2/.13,3/.34,4/.39,5/.50,6/.67,7/.72,8/.88}
   \coordinate (b\i) at ({90-360*\x}:1.2);
  \foreach \i/\j in {2/4,5/7,8/2} \draw[cv transient] (b\i)--(b\j);
  \foreach \i/\j in {0/3,3/6,6/0} \draw[cv orbit] (b\i)--(b\j);
  \foreach \i/\j in {1/4,4/7,7/1} \draw[cv second] (b\i)--(b\j);
  \foreach \i in {2,5,8} \fill[black!48] (b\i) circle (1.8pt);
  \foreach \i in {0,3,6} \fill[cvblue] ($(b\i)+(-2pt,-2pt)$) rectangle ($(b\i)+(2pt,2pt)$);
  \foreach \i in {1,4,7} \draw[draw=cvorange,fill=white,line width=.9pt] (b\i) circle (2.1pt);
  \node at (0,-1.6) {\textup{(b)} $N_{1/3}(n,r)=2$};
 \end{scope}
 \draw[-{Stealth[length=2mm]}] (3.0,-1.65)--(4.1,-1.65);
 \begin{scope}[shift={(6.0,-1.65)}]
  \foreach \i/\x in {0/0,1/.05,2/.13,3/.34,4/.39,5/.50,6/.67,7/.72,8/.88}
   \coordinate (w\i) at ({90-360*\x}:1.25);
  \foreach \i/\j in {0/1,0/2,0/3,0/6,0/7,0/8,1/2,1/3,1/4,1/7,1/8,2/3,2/4,2/8,3/4,3/5,3/6,4/5,4/6,4/7,5/6,5/7,6/7,6/8,7/8}
   \draw[draw=cvblue!52,line width=.45pt] (w\i)--(w\j);
  \foreach \i in {2,5,8} \fill[black!48] (w\i) circle (2pt);
  \foreach \i in {0,3,6} \fill[cvblue] ($(w\i)+(-2pt,-2pt)$) rectangle ($(w\i)+(2pt,2pt)$);
  \foreach \i in {1,4,7} \draw[draw=cvorange,fill=white,line width=.9pt] (w\i) circle (2.1pt);
  \node at (0,-1.6) {fill every clique};
 \end{scope}
 \draw[-{Stealth[length=2mm]}] (7.5,-1.65)--(8.8,-1.65);
 \node[font=\footnotesize,align=center] at (8.15,-2.35) {delete gray\\vertices};
 \begin{scope}[shift={(10.8,-1.65)}]
  \coordinate (o0) at (0,1.25);
  \coordinate (o4) at (0,-1.25);
  \coordinate (o1) at (-1.1,-.3);
  \coordinate (o3) at (.85,-.3);
  \coordinate (o6) at (1.1,.3);
  \coordinate (o7) at (-.85,.3);
  \foreach \i/\j/\k in {0/6/7,0/1/7,4/6/7,4/1/7}
   \fill[cvblue!8] (o\i)--(o\j)--(o\k)--cycle;
  \foreach \i/\j/\k in {0/1/3,0/3/6,4/1/3,4/3/6}
   \fill[cvblue!24] (o\i)--(o\j)--(o\k)--cycle;
  \foreach \i/\j in {0/1,0/3,0/6,1/3,1/4,3/4,3/6,4/6}
   \draw[draw=cvblue,line width=.65pt] (o\i)--(o\j);
  \foreach \i/\j in {0/7,1/7,4/7,6/7}
   \draw[draw=cvblue!65,dashed,line width=.55pt] (o\i)--(o\j);
  \foreach \i in {0,3,6} \fill[cvblue] ($(o\i)+(-2pt,-2pt)$) rectangle ($(o\i)+(2pt,2pt)$);
  \foreach \i in {1,4,7} \draw[draw=cvorange,fill=white,line width=.9pt] (o\i) circle (2.1pt);
  \node[align=center] at (.25,-1.7) {octahedral boundary $\cong S^2$};
 \end{scope}
 \node[anchor=west,font=\footnotesize] at (.2,-4.1)
 {\textcolor{cvblue}{\rule{5mm}{.8pt}} periodic orbit\quad
  \textcolor{cvorange}{\rule{5mm}{.8pt}} second periodic orbit\quad
  \textcolor{black!48}{\rule{5mm}{.6pt}} nonperiodic vertices};
\end{tikzpicture}
\caption{Two farthest-point maps with the same winding fraction and different orbit counts $N$, together with their Vietoris--Rips complexes and homotopy types.}
\label{fig:dynamics}
\end{figure}
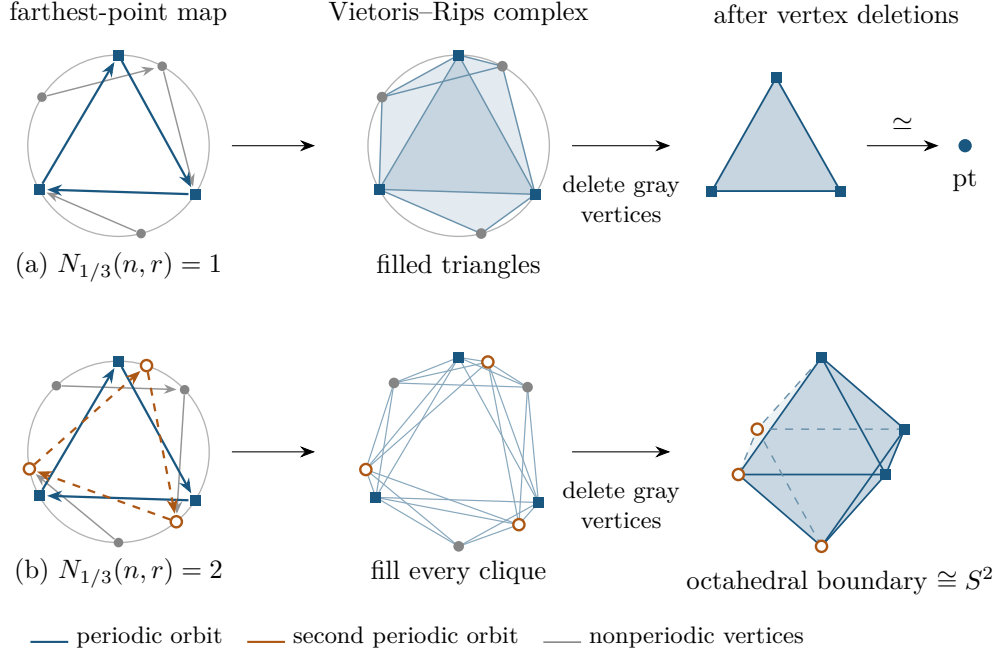

\subsection{Recovering probabilities from orbit moments}
For $q=w/\ell\in\Qn$ and $0\leq j\leq\lfloor n/\ell\rfloor$, define the binomial moments of $N_q(n,r)$ by
\begin{equation}\label{eq:binomial-moments}
 B_{q,j}(n,r)=\E_n\binom{N_q(n,r)}{j}.
\end{equation}
These are the factorial moments divided by $j!$.

\begin{lemma}\label{lem:inversion}
For $q=w/\ell\in\Qn$ and $0\leq k\leq\lfloor n/\ell\rfloor$,
\begin{align}
 \Pp_n(N_q(n,r)=k)
 &=\sum_{j=k}^{\lfloor n/\ell\rfloor}(-1)^{j-k}\binom jk B_{q,j}(n,r),\label{eq:inversion}\\
 \Pp_n(N_q(n,r)\geq1)
 &=\sum_{j=1}^{\lfloor n/\ell\rfloor}(-1)^{j-1}B_{q,j}(n,r).\label{eq:positive-count}
\end{align}
\end{lemma}
\begin{proof}
Put $N=N_q(n,r)$. Since $0\leq N\leq\lfloor n/\ell\rfloor$ almost surely and $\binom Nj=0$ for $j>N$, the binomial theorem and linearity of expectation give the polynomial identity
\[
 \E_n(1+z)^N=\sum_{j=0}^{\lfloor n/\ell\rfloor}B_{q,j}(n,r)z^j.
\]
After substituting $z=v-1$, the left-hand side is $\E_n v^N$. Expanding it according to the possible values of $N$ and expanding each $(v-1)^j$ on the right gives
\[
 \begin{aligned}
 \sum_{k=0}^{\lfloor n/\ell\rfloor}\Pp_n(N=k)v^k
 &=\sum_{j=0}^{\lfloor n/\ell\rfloor}B_{q,j}(n,r)(v-1)^j\\
 &=\sum_{j=0}^{\lfloor n/\ell\rfloor}B_{q,j}(n,r)
       \sum_{k=0}^{j}(-1)^{j-k}\binom jk v^k\\
 &=\sum_{k=0}^{\lfloor n/\ell\rfloor}
       \left(\sum_{j=k}^{\lfloor n/\ell\rfloor}(-1)^{j-k}\binom jk B_{q,j}(n,r)\right)v^k.
 \end{aligned}
\]
All sums are finite, so the change of summation order is valid. Comparing the coefficients of $v^k$ proves~\eqref{eq:inversion}. At $k=0$, this gives
\[
 \Pp_n(N=0)=1+\sum_{j=1}^{\lfloor n/\ell\rfloor}(-1)^jB_{q,j}(n,r).
\]
Taking the complementary event therefore yields
\[
 \Pp_n(N\geq1)=1-\Pp_n(N=0)
       =\sum_{j=1}^{\lfloor n/\ell\rfloor}(-1)^{j-1}B_{q,j}(n,r),
\]
which is~\eqref{eq:positive-count}.
\end{proof}

The homotopy events now have explicit orbit descriptions. An odd sphere occurs when the realized fraction lies in the corresponding open interval, regardless of the positive orbit count. A wedge of $m\geq1$ critical even spheres requires exactly $m+1$ orbits. The contractible event is the disjoint union of the events $\{N_{c_t}=1\}$ over $0\leq t\leq\lfloor(n-1)/2\rfloor$. In particular, contractibility cannot be computed from fixed points alone at scales above $1/3$.

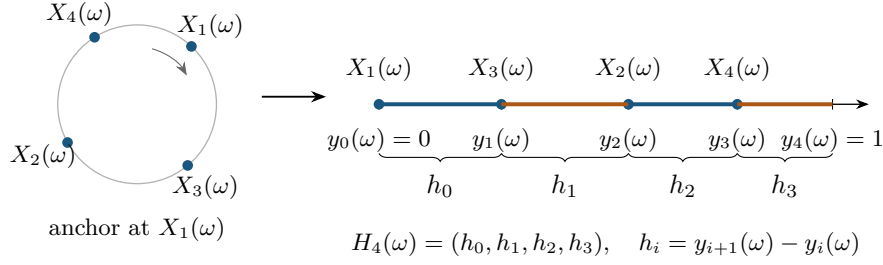
\begin{figure}[!t]
\centering
\begin{tikzpicture}[cv every]
 \begin{scope}[xshift=1.5cm,yshift=1.45cm]
  \draw[cv circle] (0,0) circle (1.05);
  \foreach \i/\x in {1/.12,2/.67,3/.39,4/.91} {
   \coordinate (p\i) at ({90-360*\x}:1.05);
   \fill[cvblue] (p\i) circle (2pt);
   \node[font=\footnotesize] at ({90-360*\x}:1.42) {$X_{\i}(\omega)$};
  }
  \draw[-{Stealth[length=1.8mm]},draw=black!60] (75:.76) arc[start angle=75,end angle=30,radius=.76];
  \node[align=center,font=\footnotesize] at (0,-1.65) {anchor at $X_1(\omega)$};
 \end{scope}
 \draw[-{Stealth[length=2mm]},line width=.7pt] (3.15,1.55)--(4.0,1.55);
 \begin{scope}[xshift=4.7cm,yshift=1.45cm,x=6cm]
  \draw[cv axis] (0,0)--(1.08,0);
  \foreach \x/\lab in {0/{y_0(\omega)=0},.27/{y_1(\omega)},.55/{y_2(\omega)},.79/{y_3(\omega)},1/{y_4(\omega)=1}} {
   \draw (\x,-.07)--(\x,.07);
   \node[below=5pt,font=\footnotesize] at (\x,0) {$\lab$};
  }
  \foreach \x/\lab in {0/{X_1(\omega)},.27/{X_3(\omega)},.55/{X_2(\omega)},.79/{X_4(\omega)}} {
   \fill[cvblue] (\x,0) circle (2pt);
   \node[above=5pt,font=\footnotesize] at (\x,0) {$\lab$};
  }
  \draw[draw=cvblue,line width=1.5pt] (0,0)--(.27,0);
  \draw[draw=cvorange,line width=1.5pt] (.27,0)--(.55,0);
  \draw[draw=cvblue,line width=1.5pt] (.55,0)--(.79,0);
  \draw[draw=cvorange,line width=1.5pt] (.79,0)--(1,0);
  \foreach \a/\b/\i in {0/.27/0,.27/.55/1,.55/.79/2,.79/1/3} {
   \draw[decorate,decoration={brace,mirror,amplitude=3pt}] (\a,-.65)--(\b,-.65);
   \node[below=4pt] at ({(\a+\b)/2},-.65) {$h_{\i}$};
  }
  \node[align=center,font=\footnotesize] at (.5,-1.85)
   {$H_4(\omega)=(h_0,h_1,h_2,h_3),\quad h_i=y_{i+1}(\omega)-y_i(\omega)$};
 \end{scope}
\end{tikzpicture}
\caption{Construction of the anchored gap map by fixing the first sample point as the anchor and taking consecutive differences of the ordered relative positions.}
\label{fig:gaps}
\end{figure}

\section{Circular gaps and exact orbit moments}\label{sec:moments}
\subsection{Coordinate maps and the gap law}
For $M\geq2$, let
\[
 \Delta_{M-1}=\left\{h\in\R_{\geq0}^{M}:
                          \sum_{i=0}^{M-1}h_i=1\right\},\qquad
 D_M=\left\{v\in\R_{\geq0}^{M-1}:
                          \sum_{i=0}^{M-2}v_i\leq1\right\}.
\]
Define the coordinate chart and its reference measure by
\begin{equation}\label{eq:coordinate-measure}
 \begin{split}
 \phi_M:D_M&\longrightarrow\Delta_{M-1},\\
 (v_0,\ldots,v_{M-2})&\longmapsto
 \left(v_0,\ldots,v_{M-2},1-\sum_{i=0}^{M-2}v_i\right),\\
 \eta_M&=(\phi_M)_*(\Leb_{M-1}|_{D_M}).
 \end{split}
\end{equation}
Here the subscript $*$ denotes pushforward. Thus $\eta_M$ is Lebesgue measure in independent gap coordinates; it is not Euclidean hypersurface measure in $\R^M$.

For an ordered sample $\omega=(\omega_1,\ldots,\omega_M)\in\Omega_M=[0,1)^M$, fix the first sample point $X_1(\omega)=\pi(\omega_1)$ as the anchor. Let $y_1(\omega)\leq\cdots\leq y_{M-1}(\omega)$ be the nondecreasing rearrangement of
\[
 \bigl((\omega_j-\omega_1)\bmod 1\bigr)_{j=2}^{M},
\]
where each relative coordinate is represented in $[0,1)$. Set $y_0(\omega)=0$ and $y_M(\omega)=1$, and define the anchored gap map by
\begin{equation}\label{eq:sample-gap-map}
 H_M:\Omega_M\longrightarrow\Delta_{M-1},\qquad
 H_M(\omega)=\bigl(y_{i+1}(\omega)-y_i(\omega)\bigr)_{i=0}^{M-1}.
\end{equation}
This defines a Borel map on the ordered sample space, with the anchor prescribed by its first coordinate. Figure~\ref{fig:gaps} illustrates this construction.

\begin{lemma}\label{lem:gap-law}
If $\nu_M=(H_M)_*\Pp_M$, then
\begin{equation}\label{eq:gap-density}
 \frac{d\nu_M}{d\eta_M}(h)=(M-1)!
 \qquad\text{for }\eta_M\text{-almost every }h\in\Delta_{M-1}.
\end{equation}
\end{lemma}
\begin{proof}
To separate the anchor from the relative positions, define
\begin{equation}\label{eq:anchor-map}
 \Theta_M:(S^1)^M\longrightarrow S^1\times(S^1)^{M-1},\qquad
 (x_1,\ldots,x_M)\longmapsto
 (x_1,x_2-x_1,\ldots,x_M-x_1).
\end{equation}
For each fixed $x_1$, translation by $-x_1$ preserves the product uniform law of the remaining coordinates. Hence $\Theta_M$ preserves product uniform measure, and the anchor is independent of the $M-1$ relative coordinates, which are independent uniforms.

For the change from ordered positions to gaps, set
\begin{equation}\label{eq:cumulative-map}
 \begin{split}
 \mathcal O_M&=\{y\in\R^{M-1}:0\leq y_1\leq\cdots\leq y_{M-1}\leq1\},\\
 \Gamma_M:\mathcal O_M&\longrightarrow\Delta_{M-1},\qquad
 \Gamma_M(y)=(y_{i+1}-y_i)_{i=0}^{M-1}\qquad (y_0=0,\ y_M=1).
 \end{split}
\end{equation}
There are $(M-1)!$ possible strict orders of the relative coordinates. Each is carried to $\mathcal O_M$ by a coordinate permutation, so the ordered density is $(M-1)!$ with respect to $\Leb_{M-1}$. The map $\Gamma_M$ is an affine bijection and has absolute Jacobian one in the independent coordinates of~\eqref{eq:coordinate-measure}. Since $H_M(\omega)=\Gamma_M(y_1(\omega),\ldots,y_{M-1}(\omega))$, this proves~\eqref{eq:gap-density}. In particular, $\eta_M(\Delta_{M-1})=1/(M-1)!$.
\end{proof}

For a fixed index set $L=\{b_1<\cdots<b_M\}\subseteq[n]$ with $M\geq2$, define the coordinate projection by
\begin{equation}\label{eq:projection}
 \pr_L:\Omega_n\longrightarrow\Omega_M,\qquad
 \pr_L(\omega)=(\omega_{b_1},\ldots,\omega_{b_M}).
\end{equation}
Then $(\pr_L)_*\Pp_n=\Pp_M$, so $H_M\circ \pr_L$ has law $\nu_M$. This is the marginal law for a fixed index set, without conditioning on its points being periodic.

\subsection{Unions of periodic orbits}
Fix $q=w/\ell\in\Qn$ and $1\leq j\leq \lfloor n/\ell\rfloor$, and put $M=j\ell$ and $s=jw$. The following specialization of the regular-core description in~\cite[Example~2.2 and Lemma~2.3]{AAM} is the counting input.

\begin{lemma}\label{lem:shift}
Let $A$ be an $M$-point subset of a distinct-point sample $X$, with increasing lifts $(z_i)_{i\in\Z}$ satisfying $z_{i+M}=z_i+1$. Then $A$ is a union of $j$ periodic orbits of the full-sample map with fraction $q$ if and only if
\begin{equation}\label{eq:shift}
 F_r(z_i)=z_{i+s}\qquad(i\in\Z).
\end{equation}
\end{lemma}
\begin{proof}
Suppose that $A$ is a union of $j$ periodic orbits of fraction $q$. Then $F_r$ restricts to a bijection of its lifted points. Write $F_r(z_i)=z_{g(i)}$. The degree-one identity gives $g(i+M)=g(i)+M$, so telescoping yields
\[
 \sum_{i=0}^{M-1}\bigl(g(i+1)-g(i)\bigr)=g(M)-g(0)=M.
\]
Order preservation and injectivity give $g(i+1)-g(i)\geq1$ for every $i\in\Z$. Hence all $M$ summands equal one. The identity $g(i+M)=g(i)+M$ extends this to every $i\in\Z$, and therefore $g(i)=i+s'$ for a fixed integer $s'$.

For the shift $i\mapsto i+s'$ modulo $M$, a return after $k\geq1$ iterations occurs if and only if $M\mid ks'$. The least such $k$ is $M/\gcd(M,s')$, so there are $\gcd(M,s')$ cycles of this period. Since an advance by $M$ indices gives one full turn, each cycle has winding number $s'/\gcd(M,s')$. Its winding fraction is thus
\[
 \frac{s'/\gcd(M,s')}{M/\gcd(M,s')}=\frac{s'}M.
\]
By the given fraction and the definitions $M=j\ell$ and $s=jw$,
\[
 \frac{s'}M=q=\frac w\ell=\frac sM.
\]
Consequently $s'=s=jw$, proving~\eqref{eq:shift}.

Conversely, a shift by $jw$ on $j\ell$ indices has $j$ cycles, since $\gcd(w,\ell)=1$. Each has period $\ell$ and winding number $w$.
\end{proof}
The selected union need not contain every periodic orbit of the sample. Once the full map is fixed, any $j$ of its $q$-orbits can be chosen, provided that at least $j$ exist.

For integers $M\geq2$ and $0\leq s<M/2$, let $A$ be an $M$-point set with increasing lifts $z_i$ satisfying $z_{i+M}=z_i+1$. Write $h=(h_0,\ldots,h_{M-1})\in\Delta_{M-1}$, where $h_i=z_{i+1}-z_i$ for $0\leq i\leq M-1$, and extend the gap indices modulo $M$. Then, for $0\leq i\leq M-1$,
\[
 z_{i+s}-z_i=\sum_{v=0}^{s-1}h_{i+v},\qquad
 z_{i+s+1}-z_i=\sum_{v=0}^{s}h_{i+v}.
\]
Applying the definition of $F_r$ in~\eqref{eq:farthest} to the point set $A$, the image of $z_i$ is $z_b$ for the largest $b\geq i$ with $z_b-z_i<r$. Since the lifts are increasing, this index is $i+s$ if and only if $z_{i+s}$ satisfies the bound and $z_{i+s+1}$ does not. Thus, for $i=0,\ldots,M-1$,
\[
 z_{i+s}-z_i<r\leq z_{i+s+1}-z_i
 \quad\Longleftrightarrow\quad
 \sum_{v=0}^{s-1}h_{i+v}<r\leq\sum_{v=0}^{s}h_{i+v}.
\]
Accordingly, define the region of gaps satisfying these successor conditions by
\begin{equation}\label{eq:successor-region}
 \mathcal S_{M,s}(r)=\left\{h\in\Delta_{M-1}:
 \sum_{v=0}^{s-1}h_{i+v}<r\leq\sum_{v=0}^{s}h_{i+v}
 \text{ for }i=0,\ldots,M-1\right\}.
\end{equation}
Gap indices are modulo $M$, and the left sum is zero when $s=0$. The anchored gap law gives
\begin{equation}\label{eq:gap-probability}
 \Pp_M\{H_M\in\mathcal S_{M,s}(r)\}
 =\nu_M(\mathcal S_{M,s}(r))
 =(M-1)!\,\eta_M(\mathcal S_{M,s}(r)).
\end{equation}

\begin{lemma}\label{lem:gap-support}
Let $M\geq2$ and $0\leq s<M/2$ be integers, and let $0<r<1/2$. Then $\nu_M(\mathcal S_{M,s}(r))>0$ if and only if $r\in I_{M,s}$.
\end{lemma}
\begin{proof}
Recall from~\eqref{eq:support} that $I_{M,s}=(s/M,(s+1)/M)\cap(0,1/2)$. Summing the successor inequalities in~\eqref{eq:successor-region} over $i=0,\ldots,M-1$ gives
\[
 s<Mr\leq s+1.
\]
If $Mr=s+1$, every right-hand inequality must be an equality. In particular, the gap vector lies in the proper affine hyperplane
\[
 h_0+\cdots+h_s=r,
\]
which has $\nu_M$-measure zero. Conversely, if $s<Mr<s+1$, the uniform gap vector satisfies all the successor inequalities strictly. A relative open neighborhood therefore satisfies them as well and has positive $\nu_M$-measure. Thus $\nu_M(\mathcal S_{M,s}(r))>0$ if and only if $r\in I_{M,s}$.
\end{proof}

\subsection{Preserving the selected orbits}\label{subsec:preserving-orbits}
For a fixed index set $L=\{b_1<\cdots<b_M\}\subseteq[n]$ of size $M\geq2$, we shall test whether the sample points indexed by $L$ form a union of periodic orbits of the full-sample map. Let $\mathcal A_L=\sigma(X_b:b\in L)$ be the $\sigma$-algebra generated by these positions, and define
\[
 E_L=\{\omega\in\Omega_n:
          H_M(\pr_L(\omega))\in\mathcal S_{M,s}(r)\}
       \in\mathcal A_L.
\]
Then $\Pp_n(E_L)=\nu_M(\mathcal S_{M,s}(r))$. The inequalities in~\eqref{eq:successor-region} force $h_{i+s}>0$ for every $i$, so the selected points are distinct on $E_L$. On this event, choose their increasing lifts $z_i$, anchored at $X_{b_1}$, and extend them periodically. Define the random subset of $S^1$ on all of $\Omega_n$ by
\begin{equation}\label{eq:avoidance-arcs}
 V_L(\omega)=
 \begin{cases}
 \displaystyle\bigcup_{i=0}^{M-1}\pi\bigl((z_{i+s},z_i+r)\bigr),
                                      &\text{if }\omega\in E_L,\\
 \varnothing,&\text{if }\omega\notin E_L.
 \end{cases}
\end{equation}
Its endpoints are $\mathcal A_L$-measurable. These are precisely the arcs into which an additional point cannot be inserted if the prescribed successors are to remain unchanged.

\begin{lemma}\label{lem:avoidance}
On $E_L$ the arcs in~\eqref{eq:avoidance-arcs} are disjoint and
\begin{equation}\label{eq:avoidance-length}
 \mu(V_L)=Mr-s.
\end{equation}
If $J_L=\{\omega:X_b(\omega)\notin V_L(\omega)\text{ for every }b\notin L\}$, then
\begin{equation}\label{eq:avoidance-conditional}
 \Pp_n(J_L\mid\mathcal A_L)=(1-\mu(V_L))^{n-M}
 \qquad\Pp_n\text{-almost surely}.
\end{equation}
When $n=M$ the conditional probability in~\eqref{eq:avoidance-conditional} is identically one.
\end{lemma}
\begin{proof}
On $E_L$, $z_{i+s}<z_i+r\leq z_{i+s+1}$. Thus the $i$th arc lies in the selected gap immediately after $z_{i+s}$. The target indices $i+s$ run through all residues, so the arcs are disjoint. Summing their lengths gives
\[
 \begin{aligned}
 \mu(V_L)
 &=\sum_{i=0}^{M-1}(z_i+r-z_{i+s})
   =\sum_{i=0}^{M-1}\left(r-\sum_{v=0}^{s-1}h_{i+v}\right)\\
 &=Mr-\sum_{v=0}^{s-1}\sum_{i=0}^{M-1}h_{i+v}
   =Mr-s\sum_{i=0}^{M-1}h_i=Mr-s.
 \end{aligned}
\]
For each $v\in\{0,\ldots,s-1\}$, addition of $v$ modulo $M$ permutes the gap indices, so the inner sum in the second line equals $\sum_{i=0}^{M-1}h_i=1$. Summing over the $s$ values of $v$ gives the last two equalities.

Given $\mathcal A_L$, the coordinates outside $L$ retain their product-uniform law. Since $V_L$ depends only on the coordinates in $L$, its complement is a fixed Borel subset of $S^1$ under this regular conditional law. The conditional product measure of $(S^1\setminus V_L)^{n-M}$ is $(1-\mu(V_L))^{n-M}$, proving~\eqref{eq:avoidance-conditional}. Outside $E_L$, the conditional probability is one because $V_L=\varnothing$.
\end{proof}
This conditional empty-arc argument is related to the calculations in~\cite[Section~3.1 and Appendix~A]{AAM}. Here the selected whole-orbit structure makes $\mu(V_L)$ independent of the selected shape.

Figure~\ref{fig:persistence} illustrates the arcs defining $V_L$ in~\eqref{eq:avoidance-arcs} within the selected gaps.

\begin{figure}[tb]
\centering
\begin{tikzpicture}[cv every]
 \begin{scope}[xshift=1.8cm,yshift=1.9cm]
  \draw[cv circle] (0,0) circle (1.4);
  \foreach \i/\x in {0/0,1/.08,2/.34,3/.46,4/.67,5/.84}
   \coordinate (a\i) at ({90-360*\x}:1.4);
  \foreach \a/\b in {0/.03,.34/.36,.67/.70}
   \draw[draw=cvred,line width=3pt] ({90-360*\a}:1.4)
     arc[start angle={90-360*\a},end angle={90-360*\b},radius=1.4];
  \foreach \i/\j in {1/2,3/4,5/1} \draw[cv transient] (a\i)--(a\j);
  \foreach \i/\j in {0/2,2/4,4/0} \draw[cv orbit] (a\i)--(a\j);
  \foreach \i in {1,3,5} \fill[black!48] (a\i) circle (1.8pt);
  \foreach \i in {0,2,4} \fill[cvblue] ($(a\i)+(-2pt,-2pt)$) rectangle ($(a\i)+(2pt,2pt)$);
  \node[align=center] at (0,1.95) {orbits and additional points};
  \node[font=\footnotesize] at (0,-1.95) {$\mu(V_L)=Mr-s$};
 \end{scope}
 \begin{scope}[xshift=4.85cm,yshift=2.25cm]
  \node[align=center] at (2.6,1.65) {a selected gap in $\mathbb R$};
  \draw[line width=.6pt] (0,0)--(5.2,0);
  \draw[cv axis] (5.2,0)--(5.6,0);
  \draw[draw=cvred,line width=3pt] (0,0)--(2.25,0);
  \foreach \x in {0,5.2} \fill[cvblue] (\x,0) circle (2.2pt);
  \draw[dashed] (2.25,-.55)--(2.25,.45);
  \node[below=6pt] at (0,0) {$z_{i+s}$};
  \node[below=18pt] at (2.25,0) {$z_i+r$};
  \node[below=6pt] at (5.2,0) {$z_{i+s+1}$};
  \draw[draw=cvred,line width=1pt] (1.0,.40)--(1.14,.54) (1.0,.54)--(1.14,.40);
  \draw[draw=cvred,dotted] (1.07,.40)--(1.07,.07);
  \fill[black!65] (3.7,0) circle (1.9pt);
  \node[above=9pt,font=\footnotesize,text=cvred] at (1.07,.54) {changes successor};
  \node[above=9pt,font=\footnotesize] at (4.0,.54) {preserves successor};
  \draw[decorate,decoration={brace,mirror,amplitude=4pt}] (0,-1.35)--(2.25,-1.35);
  \node[below=7pt,align=center] at (1.125,-1.35) {$r-(z_{i+s}-z_i)$};
 \end{scope}
\end{tikzpicture}
\caption{The sample and successor relations from Figure~\ref{fig:dynamics}(a), with $V_L$ in red. Gray nonperiodic points lie outside $V_L$; the enlarged gap shows where an additional point changes or preserves a selected successor.}
\label{fig:persistence}
\end{figure}
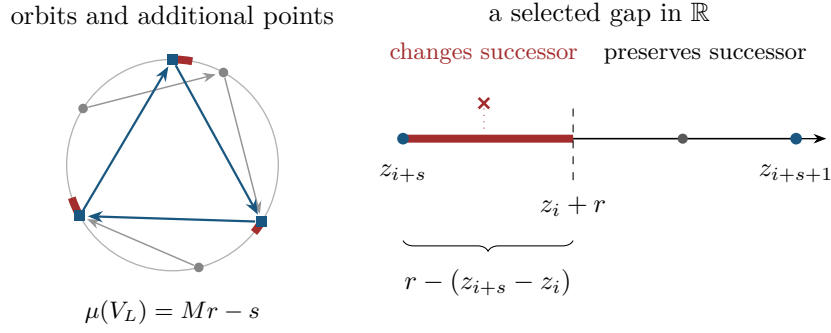

Related moment expansions for spacing counts and circular coverage inequalities appear in~\cite[Examples~2--3]{HLpre}. The next proposition computes these binomial moments from the gap measure and the constant avoidance factor in Lemma~\ref{lem:avoidance}.

\begin{proposition}\label{prop:moments}
For $n\geq2$, $0<r<1/2$, $q=w/\ell\in\Qn$ and $1\leq j\leq \lfloor n/\ell\rfloor$, put $M=j\ell$ and $s=jw$. If $M\geq2$, the binomial moments in~\eqref{eq:binomial-moments} satisfy
\begin{equation}\label{eq:orbit-moment}
 B_{q,j}(n,r)=
 \begin{cases}
  \displaystyle\binom nM(1-Mr+s)^{n-M}
       \nu_M(\mathcal S_{M,s}(r)),
       &\text{if }r\in I_{M,s},\\[3pt]
  0,&\text{if }r\notin I_{M,s}.
 \end{cases}
\end{equation}
For $q=0$, the binomial moments are given explicitly by
\[
 B_{0,M}(n,r)=\binom nM(1-Mr)_+^{n-1},
 \qquad 1\leq M\leq n.
\]
\end{proposition}
\begin{proof}
For $M\geq2$ and each $M$-element index set $L\subseteq[n]$, use the increasing lifts $(z_i)_{i\in\Z}$ of the sample points indexed by $L$ from~\eqref{eq:avoidance-arcs}, with $z_{i+M}=z_i+1$. On the distinct-point locus, define
\[
 I_L(\omega)=
 \begin{cases}
 1,&\text{if }F_r(z_i)=z_{i+s}\text{ for every }i=0,\ldots,M-1,\\
 0,&\text{otherwise}.
 \end{cases}
\]
Here $F_r$ is the map determined by the full sample $\Xn(\omega)$. Set $I_L=0$ when the sample points are not all distinct. By Lemma~\ref{lem:shift}, $I_L$ indicates whether the points indexed by $L$ form a union of $j$ periodic orbits with fraction $q$. The definitions of $E_L$ and $J_L$ then give
\[
 I_L=\ind_{E_L\cap J_L}=\ind_{E_L}\ind_{J_L}
 \qquad\Pp_n\text{-almost surely}.
\]
Here $\ind_E$ denotes the indicator of the event $E$.
Moreover,
\begin{equation}\label{eq:union-count}
 \binom{N_q(n,r)}j=\sum_{\substack{L\subseteq[n]\\|L|=M}}I_L
 \qquad\Pp_n\text{-almost surely}.
\end{equation}
Indeed, choosing $j$ distinct periodic orbits determines a union of exactly $j\ell=M$ sample indices. Conversely, such a union determines its selected collection of orbits. No ordering or starting point on an orbit is chosen.

For $r\in I_{M,s}$, we have $Mr<s+1$, hence $1-Mr+s>0$. Lemma~\ref{lem:avoidance} implies
\[
 \E_n(I_L\mid\mathcal A_L)
       =\ind_{E_L}(1-Mr+s)^{n-M}.
\]
Taking expectations and applying the tower property gives
\[
 \begin{aligned}
 \E_n I_L
 &=\E_n\bigl[\E_n(I_L\mid\mathcal A_L)\bigr]\\
 &=\E_n\bigl[\ind_{E_L}(1-Mr+s)^{n-M}\bigr]\\
 &=(1-Mr+s)^{n-M}\Pp_n(E_L)\\
 &=(1-Mr+s)^{n-M}\nu_M(\mathcal S_{M,s}(r)).
 \end{aligned}
\]
Taking expectations in~\eqref{eq:union-count} therefore gives
\[
 \begin{aligned}
 B_{q,j}(n,r)
 &=\E_n\sum_{\substack{L\subseteq[n]\\|L|=M}}I_L
   =\sum_{\substack{L\subseteq[n]\\|L|=M}}\E_n I_L\\
 &=\sum_{\substack{L\subseteq[n]\\|L|=M}}
        (1-Mr+s)^{n-M}\nu_M(\mathcal S_{M,s}(r))\\
 &=\binom nM(1-Mr+s)^{n-M}\nu_M(\mathcal S_{M,s}(r)).
 \end{aligned}
\]
A sample with $k\geq j$ periodic orbits of fraction $q$ contributes $\binom{k}{j}$ to this count, since any $j$ of those orbits may be selected.

For $r\notin I_{M,s}$, Lemma~\ref{lem:gap-support} gives $\Pp_n(E_L)=\nu_M(\mathcal S_{M,s}(r))=0$. Since $I_L\leq\ind_{E_L}$ almost surely, every $\E_n I_L$ is zero, and hence $B_{q,j}(n,r)=0$. This proves~\eqref{eq:orbit-moment} for every admissible $s$ when $M\geq2$.

Now suppose $q=0$. Since $q=w/\ell$ is reduced, we have $w=0$ and $\ell=1$, so $s=0$ and $j=M$. For $M\geq2$, the successor inequalities in~\eqref{eq:successor-region} reduce to $h_i\geq r$ for every $i=0,\ldots,M-1$. Hence
\[
 \mathcal S_{M,0}(r)
 =\{h\in\Delta_{M-1}:h_i\geq r
   \text{ for every }i=0,\ldots,M-1\}.
\]
In this case $I_{M,0}=(0,1/M)\cap(0,1/2)$, so $r\in I_{M,0}$ if and only if $Mr<1$. For $Mr<1$, the change of variables
\[
 h_i=r+(1-Mr)v_i,\qquad i=0,\ldots,M-1,
\]
maps $\Delta_{M-1}$ bijectively onto $\mathcal S_{M,0}(r)$. It scales each of the $M-1$ independent gap coordinates by $1-Mr$, so the constant density of $\nu_M$ gives
\[
 \nu_M(\mathcal S_{M,0}(r))=(1-Mr)^{M-1}.
\]
Substituting into~\eqref{eq:orbit-moment} yields
\[
 \begin{aligned}
 B_{0,M}(n,r)
 &=\binom nM(1-Mr)^{n-M}(1-Mr)^{M-1}\\
 &=\binom nM(1-Mr)^{n-1}.
 \end{aligned}
\]
For $Mr\geq1$, we have $r\notin I_{M,0}$, so~\eqref{eq:orbit-moment} gives $B_{0,M}(n,r)=0$. Together these prove the stated zero-winding formula for $M\geq2$.

Finally, if $M=1$, the integer condition $0\leq s<M/2$ forces $s=0$, and $M=j\ell$ gives $j=\ell=1$. Each $X_b$ is fixed if and only if the other $n-1$ points avoid its clockwise arc of length $r$. Summing the fixed-point probabilities over the $n$ sampled points gives
\[
 B_{0,1}(n,r)=n(1-r)^{n-1},
\]
which completes the proof.
\end{proof}

\section{Exact evaluation for all real scales}\label{sec:evaluation}
We evaluate $\nu_M(\mathcal S_{M,s}(r))$ for $M\geq3$ and $1\leq s<M/2$, first establishing its real-parameter geometry, then using lattice counts and interpolation. Combining this evaluation with Proposition~\ref{prop:moments} proves Theorem~\ref{thm:finite}.

General methods for exact probabilities of linear inequalities in uniform spacings are already available in Huffer and Lin~\cite{HL97,HL01}, with further circular-spacing and coverage algorithms in Lin and Lu~\cite{LL}. Thus generic exact computability and piecewise-polynomial dependence are prior; the calculation below uses the particular difference constraints of the orbit region to obtain explicit coefficient formulas and a denominator bound for the possible breakpoints.

\subsection{Difference constraints and vertex coordinates}
Define
\begin{equation}\label{eq:polytope}
 Q_{M,s}(r)=\left\{h\in\Delta_{M-1}:
 \sum_{v=0}^{s-1}h_{i+v}\leq r\leq\sum_{v=0}^{s}h_{i+v}
                     \text{ for }i=0,\ldots,M-1\right\}.
\end{equation}
This is obtained by weakening the strict inequalities in~\eqref{eq:successor-region}. We do not identify it with the closure of that region at every scale: a nonempty lower-dimensional polytope can remain when the strict region is empty. Since the changed boundaries are proper affine hyperplanes,
\begin{equation}\label{eq:polytope-measure}
 \nu_M(\mathcal S_{M,s}(r))=\nu_M(Q_{M,s}(r))
 =(M-1)!\,\eta_M(Q_{M,s}(r)),\qquad 0<r<1/2.
\end{equation}

We now express the defining inequalities of $Q_{M,s}(r)$ in the cumulative coordinates given by the inverse of the gap transformation~\eqref{eq:cumulative-map}. For $h\in Q_{M,s}(r)\subseteq\Delta_{M-1}$, set
\[
 y_0=0,\qquad y_i=\sum_{v=0}^{i-1}h_v\quad(1\leq i\leq M).
\]
Then $y_M=1$ and $y_{i+1}-y_i=h_i$ for $0\leq i\leq M-1$, so
\[
 \Gamma_M^{-1}(h)=(y_1,\ldots,y_{M-1}).
\]
Extend the gap indices modulo $M$ and the cumulative coordinates by $y_{i+M}=y_i+1$ for $i\in\Z$. Telescoping gives, for every $k\geq1$,
\[
 \sum_{v=0}^{k-1}h_{i+v}
 =\sum_{v=0}^{k-1}(y_{i+v+1}-y_{i+v})
 =y_{i+k}-y_i.
\]
Taking $k=s$ and $k=s+1$, the two bounds in~\eqref{eq:polytope} become
\[
 y_{i+s}-y_i\leq r\leq y_{i+s+1}-y_i,
 \qquad i=0,\ldots,M-1.
\]
Equivalently, they are $y_{i+s}-y_i\leq r$ and $y_i-y_{i+s+1}\leq-r$. Gap nonnegativity gives the additional inequalities $y_i-y_{i+1}\leq0$.

After reducing the endpoint indices modulo $M$, the three families have $b-a$ congruent to $s$, $-(s+1)$ and $-1$, respectively. Since $1\leq s<M/2$, two families can give the same ordered index pair only when $M=2s+1$, in which case the first two families coincide in their endpoint pairs. For each repeated pair, their bounds are $C+r$ and $C+1-r$ for an integer $C$. Throughout $0<r<1/2$, the former is strictly smaller, so we discard the latter. This removes the redundant bound without changing the polytope and leaves at most one defining inequality for each ordered pair.

\begin{samepage}
For each retained ordered index pair $(a,b)$, write its defining inequality as
\begin{equation}\label{eq:difference-rows}
 y_b-y_a\leq A_{ab}+B_{ab}r,
 \qquad A_{ab}\in\Z,\quad B_{ab}\in\{-1,0,1\}.
\end{equation}
The constants $A_{ab},B_{ab}$ are fixed by the retained inequality and do not depend on $r$. Here $A_{ab}$ records the periodic integer shift, while $B_{ab}=1,-1,0$ for the first bound, the second bound and gap nonnegativity, respectively.
\end{samepage}
The matrices of $\Gamma_M$ and its inverse in independent coordinates have integer entries and determinant of absolute value one. Related difference-constraint structures are discussed in~\cite[Sections~2.3 and~3.1]{LP}.

\begin{lemma}\label{lem:vertices}
Every vertex of $Q_{M,s}(r)$ has coordinates in $\Z+\Z r$. If $r=A/B$ is rational, all vertices of $BQ_{M,s}(A/B)$ have integer coordinates.
\end{lemma}
\begin{proof}
Fix a vertex $h^\ast$ of $Q_{M,s}(r)$, and let
\[
 y^\ast=\Gamma_M^{-1}(h^\ast)
       =(y_1^\ast,\ldots,y_{M-1}^\ast),
 \qquad y_0^\ast=0,\quad y_M^\ast=1.
\]
Since $\Gamma_M^{-1}$ is an affine bijection, $y^\ast$ is a vertex of $\Gamma_M^{-1}(Q_{M,s}(r))$. We work with the $M-1$ free coordinates $y_1,\ldots,y_{M-1}$.

A defining inequality in~\eqref{eq:difference-rows} is active at $y^\ast$ if equality holds there. Here the normal of a constraint is the vector of coefficients of $y_1,\ldots,y_{M-1}$ in its left-hand side; it is perpendicular to the boundary hyperplane where equality holds. The normals of the active constraints must span $\R^{M-1}$. Otherwise, there would be a nonzero vector $d\in\R^{M-1}$ orthogonal to every active normal. Moving from $y^\ast$ to $y^\ast\pm\varepsilon d$ preserves all active equalities. Since the remaining finitely many inequalities are strict at $y^\ast$, both perturbed points still satisfy them for sufficiently small $\varepsilon>0$. This would express $y^\ast$ as the midpoint of two distinct points of the polytope, contradicting that it is a vertex. We may therefore select $M-1$ active constraints with linearly independent normals.

Construct an auxiliary graph $G$ with nodes $v_0,\ldots,v_{M-1}$, one for each coordinate index. For each selected constraint with ordered index pair $(a,b)$, write its equality as
\[
 y_b^\ast-y_a^\ast=A_{ab}+B_{ab}r,
 \qquad A_{ab}\in\Z,\quad B_{ab}\in\{-1,0,1\},
\]
and draw an edge from $v_a$ to $v_b$. Thus the graph records the selected coordinate-difference equalities at the fixed polytope vertex $y^\ast$.

These $M-1$ edges contain no undirected cycle. Indeed, summing the corresponding linear forms $y_b-y_a$ around a cycle, with signs according to their orientations, gives the zero linear form. This would give a linear dependence among the selected normals. In particular, independence also excludes parallel edges, whose normals would be proportional. Hence $G$ is a forest. A forest with $M$ nodes and $c$ connected components has $M-c$ edges; since $G$ has $M-1$ edges, it has one component. Therefore $G$ is a spanning tree on $v_0,\ldots,v_{M-1}$.

For each $i=1,\ldots,M-1$, let $P_i$ be the unique path from $v_0$ to $v_i$. Along this path, add an edge equality when traversing its chosen orientation and subtract it otherwise. Since $y_0^\ast=0$, this yields
\[
 y_i^\ast
 =\sum_{(a,b)\in P_i}\varepsilon_{i,ab}(A_{ab}+B_{ab}r)
 \in\Z+\Z r,
 \qquad \varepsilon_{i,ab}\in\{-1,1\}.
\]
Returning to gap coordinates, we have
\[
 h_i^\ast=y_{i+1}^\ast-y_i^\ast\quad(0\leq i\leq M-2),
 \qquad h_{M-1}^\ast=1-y_{M-1}^\ast.
\]
Thus every coordinate of $h^\ast$ belongs to $\Z+\Z r$. If $r=A/B$ is rational, multiplying these coordinates by $B$ gives integers. Scaling by $B$ carries the vertices of $Q_{M,s}(A/B)$ bijectively to those of $BQ_{M,s}(A/B)$, proving the second assertion.
\end{proof}

\subsection{Intervals of polynomial dependence}
The elements of $\mathcal R_M$ in~\eqref{eq:support} are candidate values of the scale parameter $r$ at which the structure of $Q_{M,s}(r)$ can change.

\begin{lemma}\label{lem:polynomial}
The function $r\mapsto\nu_M(Q_{M,s}(r))$ is continuous on $(0,1/2)$ and polynomial of degree at most $M-1$, with rational coefficients, on every interval between consecutive elements of $\mathcal R_M$.
\end{lemma}

\begin{proof}
By the proof of Lemma~\ref{lem:vertices}, every vertex at a fixed scale is determined by $M-1$ active constraints whose associated edges form a spanning tree. To follow these points as $r$ varies, select $M-1$ inequalities from~\eqref{eq:difference-rows} whose edges form a spanning tree $\mathcal T$, and keep this selection fixed. With $y_{\mathcal T,0}(r)=0$, imposing the selected inequalities as equalities gives a unique solution $y_{\mathcal T}(r)$. Its coordinates have the form $y_{\mathcal T,i}(r)=\alpha_i+\beta_i r$ with $\alpha_i,\beta_i\in\Z$, by the same path-sum argument as in Lemma~\ref{lem:vertices}. Set $y_{\mathcal T,M}(r)=1$ and define the corresponding candidate in gap coordinates by
\[
 h_{\mathcal T}(r)=\bigl(y_{\mathcal T,i+1}(r)-y_{\mathcal T,i}(r)\bigr)_{i=0}^{M-1}.
\]
If $\mathcal T$ was selected from a vertex $h^\ast$ at scale $r_0$, then $h_{\mathcal T}(r_0)=h^\ast$. At other scales, the candidate need not satisfy all the defining inequalities.

We test this candidate against every retained inequality. For any ordered index pair $(a,b)$ retained in~\eqref{eq:difference-rows}, define the slack of its inequality at this candidate by
\begin{equation}\label{eq:constraint-difference}
 \begin{aligned}
 D_{\mathcal T,ab}(r)
 &=A_{ab}+B_{ab}r-\bigl(y_{\mathcal T,b}(r)-y_{\mathcal T,a}(r)\bigr)\\
 &=\alpha_{\mathcal T,ab}+\beta_{\mathcal T,ab}r,
 \qquad \alpha_{\mathcal T,ab},\beta_{\mathcal T,ab}\in\Z.
 \end{aligned}
\end{equation}
The selected tree inequalities have identically zero slack. Moreover,
\[
 h_{\mathcal T}(r)\in Q_{M,s}(r)
 \quad\Longleftrightarrow\quad
 D_{\mathcal T,ab}(r)\geq0\ \text{for every retained pair }(a,b).
\]
Whenever these conditions hold, $h_{\mathcal T}(r)$ is a vertex: the selected $M-1$ constraints remain active, have independent normals and uniquely determine the candidate. Its active constraints are precisely the retained pairs for which $D_{\mathcal T,ab}(r)=0$.

For any retained pair $(a,b)$, the path from $v_a$ to $v_b$ in $\mathcal T$ contains at most $M-1$ edges. Summing their equalities with the appropriate signs expresses $y_{\mathcal T,b}(r)-y_{\mathcal T,a}(r)$ as a path sum. Each edge has an $r$-coefficient in $\{-1,0,1\}$, and subtracting this sum from the bound for $(a,b)$ gives $|\beta_{\mathcal T,ab}|\leq M$. Thus, if $\beta_{\mathcal T,ab}\ne0$, the only zero of $D_{\mathcal T,ab}$ has reduced denominator at most $M$. A constant slack, including an identically zero one, produces no isolated transition.

Now repeat the construction for every spanning tree that can be selected from the retained inequalities. There are finitely many such selections, and at each scale every vertex occurs among their feasible candidates by Lemma~\ref{lem:vertices}.

Fix an open interval $\mathcal J$ between consecutive elements of $\mathcal R_M$. The denominator bound above shows that no nonconstant slack has a zero inside $\mathcal J$. Thus each slack is strictly positive, strictly negative or identically zero throughout $\mathcal J$. For each fixed tree, its candidate is therefore either feasible throughout $\mathcal J$ or infeasible throughout $\mathcal J$. For a feasible candidate, the active pairs are exactly those for which $D_{\mathcal T,ab}$ vanishes identically on $\mathcal J$. Consequently, both the list of feasible candidate functions and their active pairs remain fixed, while their coordinates may vary with $r$.

Different trees may produce the same vertex, so we must also remove duplicate candidates. Suppose $\mathcal T_1$ and $\mathcal T_2$ give feasible candidates that coincide at some $r_0\in\mathcal J$. Since the cumulative coordinates are determined by the gaps,
\[
 h_{\mathcal T_1}(r_0)=h_{\mathcal T_2}(r_0)
 \quad\Longrightarrow\quad
 y_{\mathcal T_1}(r_0)=y_{\mathcal T_2}(r_0).
\]
For any pair $(a,b)$ selected for $\mathcal T_2$, this gives
\[
 \begin{aligned}
 D_{\mathcal T_1,ab}(r_0)
 &=A_{ab}+B_{ab}r_0-\bigl(y_{\mathcal T_1,b}(r_0)-y_{\mathcal T_1,a}(r_0)\bigr)\\
 &=A_{ab}+B_{ab}r_0-\bigl(y_{\mathcal T_2,b}(r_0)-y_{\mathcal T_2,a}(r_0)\bigr)
 =D_{\mathcal T_2,ab}(r_0)=0.
 \end{aligned}
\]
The last equality holds because this pair belongs to the selected equalities for $\mathcal T_2$.

By the constant-sign property on $\mathcal J$, a slack that vanishes at $r_0$ must vanish identically there. Hence, for every pair selected for $\mathcal T_2$,
\[
 D_{\mathcal T_1,ab}(r)=0,
 \qquad
 y_{\mathcal T_1,b}(r)-y_{\mathcal T_1,a}(r)=A_{ab}+B_{ab}r
 \quad(r\in\mathcal J).
\]
Thus $y_{\mathcal T_1}(r)$ solves all the selected equalities for $\mathcal T_2$, with the same normalization $y_{\mathcal T_1,0}(r)=0$. That system has a unique solution, so
\[
 y_{\mathcal T_1}(r)=y_{\mathcal T_2}(r)
 \quad\Longrightarrow\quad
 h_{\mathcal T_1}(r)=h_{\mathcal T_2}(r)
 \qquad(r\in\mathcal J).
\]
Two feasible candidate functions therefore either agree throughout $\mathcal J$ or never coincide inside it. Counting identical functions once gives a fixed list of distinct vertex functions on $\mathcal J$.

For a collection $\mathcal C$ of retained ordered pairs, write $y=\Gamma_M^{-1}(h)$ and set
\[
 F_{\mathcal C}(r)
 =\bigl\{h\in Q_{M,s}(r):
 y_b-y_a=A_{ab}+B_{ab}r\ \text{for every }(a,b)\in\mathcal C\bigr\}.
\]
Whenever this set is nonempty, it is a face of $Q_{M,s}(r)$, and every nonempty face arises in this way. A vertex $h_{\mathcal T}(r)$ belongs to $F_{\mathcal C}(r)$ if and only if all the constraints indexed by $\mathcal C$ are active there. This membership is constant on $\mathcal J$, since each corresponding slack either vanishes identically or never vanishes on that interval. As every nonempty face contains a vertex, whether $F_{\mathcal C}(r)$ is empty is also independent of $r\in\mathcal J$.

Every nonempty face is the convex hull of its vertices. Hence the subsets of vertex functions that determine faces, and thus the vertex--face incidence, are fixed throughout $\mathcal J$. This fixes which vertex functions give the endpoints of each edge and which belong to each higher-dimensional face, while their coordinates continue to vary with $r$. The dimension of a nonempty face is one less than the number of faces in a maximal chain from a vertex to that face, so these fixed inclusions also keep all face dimensions constant throughout $\mathcal J$.

Order this finite list of vertex functions once and use its pulling triangulation: successively cone from the first vertex over the recursively triangulated faces not containing it. Because the face incidences remain fixed, the same list of vertex tuples triangulates the polytope for all $r\in\mathcal J$. In a full-dimensional simplex of that list, write its independent gap coordinates as $v_0(r),\ldots,v_{M-1}(r)\in\R^{M-1}$. Its coordinate measure is
\begin{equation}\label{eq:simplex-determinant}
 \frac{1}{(M-1)!}
 \left|\det\bigl(v_1(r)-v_0(r),\ldots,v_{M-1}(r)-v_0(r)\bigr)\right|.
\end{equation}
Each entry is integer-affine in $r$. The determinant has degree at most $M-1$ and cannot vanish inside $\mathcal J$, since these tuples remain simplices of the pulling triangulation. Its sign is consequently constant. Summing~\eqref{eq:simplex-determinant}, whose overlaps have zero coordinate measure, and multiplying by $(M-1)!$ as in~\eqref{eq:polytope-measure} proves polynomiality with rational coefficients. If the polytope is empty or lower-dimensional throughout $\mathcal J$, its ambient coordinate measure is identically zero.

For continuity, fix $r_0\in(0,1/2)$ and let $(r_k)$ be any sequence in $(0,1/2)$ converging to $r_0$. Suppose $h\in\Delta_{M-1}$ satisfies
\begin{equation}\label{eq:continuity-nonboundary}
 \sum_{v=0}^{s-1}h_{i+v}\ne r_0
 \quad\text{and}\quad
 \sum_{v=0}^{s}h_{i+v}\ne r_0
 \qquad(i=0,\ldots,M-1).
\end{equation}
For this fixed $h$, the finitely many gap sums have positive distances from $r_0$. Thus every inequality defining $\mathcal S_{M,s}(r)$ in~\eqref{eq:successor-region} has the same truth value at $r=r_k$ and at $r=r_0$ for all sufficiently large $k$. The corresponding indicators are therefore eventually equal.

The set of $h\in\Delta_{M-1}$ for which equality holds in at least one comparison in~\eqref{eq:continuity-nonboundary} is contained in a finite union of proper affine hyperplanes. Indeed, each displayed sum uses a nonempty proper subset of the gaps, so it is nonconstant on the affine hull of $\Delta_{M-1}$. These hyperplanes have $\eta_M$-measure zero and hence $\nu_M$-measure zero by~\eqref{eq:gap-density}. Consequently,
\[
 \mathbf1_{\mathcal S_{M,s}(r_k)}(h)
 \longrightarrow
 \mathbf1_{\mathcal S_{M,s}(r_0)}(h)
 \qquad\text{for }\nu_M\text{-almost every }h.
\]
Since $\nu_M$ is a probability measure and all these indicators are bounded by $1$, bounded convergence and~\eqref{eq:polytope-measure} give
\[
 \begin{aligned}
 \lim_{k\to\infty}\nu_M(Q_{M,s}(r_k))
 &=\lim_{k\to\infty}\int_{\Delta_{M-1}}
   \mathbf1_{\mathcal S_{M,s}(r_k)}(h)\,d\nu_M(h)\\
 &=\int_{\Delta_{M-1}}
   \mathbf1_{\mathcal S_{M,s}(r_0)}(h)\,d\nu_M(h)\\
 &=\nu_M(Q_{M,s}(r_0)).
 \end{aligned}
\]
As the sequence was arbitrary, this proves continuity at every $r_0\in(0,1/2)$, including the algebraic support endpoints lying in that interval.
\end{proof}
The triangulation proves polynomiality; it need not be constructed by the arithmetic evaluator below. The use of difference constraints additionally identifies the explicit denominator bound, beyond general spacing-integration methods~\cite{HL01}.

\subsection{Rational scales and lattice-point counts}
Take $r=A/B\in I_{M,s}$ in lowest terms. The map
\begin{equation}\label{eq:dilation-projection}
 Q_{M,s}(A/B)\longrightarrow\R^{M-1},\qquad
 h\longmapsto(Bh_0,\ldots,Bh_{M-2})
\end{equation}
has image an integral polytope $\widehat Q$. To see the lattice exactly, on the affine hyperplane $\sum_{i=0}^{M-1}d_i=B$, deleting the last coordinate has inverse
\[
 (d_0,\ldots,d_{M-2})\longmapsto
 \left(d_0,\ldots,d_{M-2},B-\sum_{i=0}^{M-2}d_i\right).
\]
It identifies the integer points with $\Z^{M-1}$. The same statement holds with $B$ replaced by $kB$ for any integer $k\geq0$.

\begin{lemma}\label{lem:finite-difference}
For $A/B\in I_{M,s}$, the gap measure $\nu_M(Q_{M,s}(A/B))$ is given by~\eqref{eq:rational-coefficient}.
\end{lemma}
\begin{proof}
The uniform gap vector satisfies all defining inequalities strictly because $A/B\in I_{M,s}$, as in the proof of Lemma~\ref{lem:gap-support}. Thus $\widehat Q$ has nonempty interior, and Lemma~\ref{lem:vertices} gives its integrality. Ehrhart's theorem~\cite[Theorem~3.8 and Corollary~3.20]{BR} gives a polynomial
\[
 E(k)=\#(k\widehat Q\cap\Z^{M-1}),\qquad k\in\Z_{\geq0},
\]
of degree $M-1$, with $E(0)=1$ and the coefficient of $k^{M-1}$ given by
\[
 [k^{M-1}]E(k)=B^{M-1}\eta_M(Q_{M,s}(A/B)).
\]
Recall from~\eqref{eq:count} that, for every integer $k\geq0$,
\[
 C_{M,s}(kB,kA)
 =\#\left\{d\in\Z_{\geq0}^{M}:
 \sum_{i=0}^{M-1}d_i=kB,\quad
 \sum_{v=0}^{s-1}d_{i+v}\leq kA\leq\sum_{v=0}^{s}d_{i+v}
 \quad(i=0,\ldots,M-1)\right\}.
\]
The indices of $d$ are interpreted modulo $M$. For $k\geq1$, substituting $d_i=kBh_i$ converts the defining inequalities of $Q_{M,s}(A/B)$ into precisely these conditions. Deleting the last coordinate identifies these integer vectors bijectively with $k\widehat Q\cap\Z^{M-1}$: the missing coordinate is the integer $d_{M-1}=kB-\sum_{i=0}^{M-2}d_i$. At $k=0$, both counts are $1$. Hence $E(k)=C_{M,s}(kB,kA)$ for every $k\geq0$.

Define $(\Delta f)(k)=f(k+1)-f(k)$ and its iterates by $\Delta^0f=f$ and $\Delta^{v+1}f=\Delta(\Delta^vf)$. Repeated expansion of this definition gives
\[
 (\Delta^{M-1}E)(0)
 =\sum_{k=0}^{M-1}(-1)^{M-1-k}\binom{M-1}{k}E(k).
\]
To compute this difference from the coefficients of $E$, the binomial theorem gives
\[
 \Delta(k^j)=(k+1)^j-k^j
 =\sum_{u=0}^{j-1}\binom ju k^u
 \qquad(j\geq1).
\]
Applying $\Delta$ to a nonconstant polynomial lowers its degree by one and multiplies its leading coefficient by its degree. For $j\geq1$, the successive leading-coefficient factors are $j,j-1,\ldots,1$, and after $j$ differences only a constant remains. Thus
\[
 \Delta^j(k^j)=j(j-1)\cdots1=j!,
 \qquad
 \Delta^{j+1}(k^j)=\Delta(j!)=0
 \quad(j\geq1).
\]
Together with $\Delta(1)=0$, this gives
\[
 \Delta^{M-1}(k^j)
 =\begin{cases}
  0,&\text{if }0\leq j<M-1,\\
  (M-1)!,&\text{if }j=M-1.
 \end{cases}
\]
By linearity,
\[
 (\Delta^{M-1}E)(0)
 =\sum_{j=0}^{M-1}\bigl([k^j]E(k)\bigr)
   \left.\Delta^{M-1}(k^j)\right|_{k=0}
 =(M-1)!\,[k^{M-1}]E(k).
\]
Using~\eqref{eq:polytope-measure} and the leading-coefficient identity above, we obtain the expression in~\eqref{eq:rational-coefficient}:
\[
 \begin{aligned}
 \nu_M(\mathcal S_{M,s}(A/B))
 &=\nu_M(Q_{M,s}(A/B))
  =(M-1)!\,\eta_M(Q_{M,s}(A/B))\\
 &=\frac{(M-1)!}{B^{M-1}}[k^{M-1}]E(k)
  =\frac{1}{B^{M-1}}(\Delta^{M-1}E)(0)\\
 &=\frac{1}{B^{M-1}}\sum_{k=0}^{M-1}
   (-1)^{M-1-k}\binom{M-1}{k}E(k)\\
 &=\frac{1}{B^{M-1}}\sum_{k=0}^{M-1}
   (-1)^{M-1-k}\binom{M-1}{k}C_{M,s}(kB,kA).
 \end{aligned}
\]
\end{proof}

At $r=s/M$ or $r=(s+1)/M$ within the scale range, the weak polytope is nonempty but lower-dimensional. Indeed, the uniform gap vector belongs to it, and summing the inequalities forces a proper gap sum to be constant. Its $(M-1)$st Ehrhart difference is zero. Outside the closed algebraic support the polytope is empty; the formal count still has $C_{M,s}(0,0)=1$, so the same finite-difference expression must not be used there. The prescription that the arithmetic coefficient is zero outside $I_{M,s}$ prevents this spurious contribution.

\subsection{Interpolation and completion of the finite law}\label{sec:finite-completion}
\begin{proposition}\label{prop:gap-evaluation}
For $M\geq3$, $1\leq s<M/2$ and $0<r<1/2$, the coefficient $P_{M,s}(r)$ defined in the Introduction satisfies
\begin{equation}\label{eq:coefficient-gap-identity}
 P_{M,s}(r)=\nu_M(Q_{M,s}(r))
          =\nu_M(\mathcal S_{M,s}(r)).
\end{equation}
\end{proposition}
\begin{proof}
For consecutive $\alpha=A_-/B_-<\beta=A_+/B_+$ in $\mathcal R_M$, the numbers $\rho_1,\ldots,\rho_M$ in~\eqref{eq:gap-interpolation} satisfy
\[
 \rho_k-\alpha=
 \frac{A_+B_--A_-B_+}{B_-(kB_-+B_+)}>0,
 \qquad
 \beta-\rho_k=
 \frac{k(A_+B_--A_-B_+)}{B_+(kB_-+B_+)}>0.
\]
They are strictly decreasing, hence distinct. Every endpoint of $I_{M,s}$ lying in $(0,1/2)$ belongs to $\mathcal R_M$, so if $(\alpha,\beta)$ meets $I_{M,s}$, it is contained in $I_{M,s}$.

Suppose $(\alpha,\beta)\subset I_{M,s}$. At each of the $M$ rational points $\rho_k$, Lemma~\ref{lem:finite-difference} gives
\[
 P_{M,s}(\rho_k)=\nu_M(Q_{M,s}(\rho_k))
 \qquad(k=1,\ldots,M).
\]
By Lemma~\ref{lem:polynomial}, the gap measure on $(\alpha,\beta)$ is a polynomial in $r$ of degree at most $M-1$. Its Lagrange interpolation formula at these points is
\[
 \nu_M(Q_{M,s}(r))
 =\sum_{k=1}^{M}P_{M,s}(\rho_k)
   \prod_{\substack{1\leq v\leq M\\v\ne k}}
     \frac{r-\rho_v}{\rho_k-\rho_v},
 \qquad r\in(\alpha,\beta).
\]
Indeed, each product on the right has $M-1$ linear factors in $r$, so the right-hand side is a polynomial of degree at most $M-1$. At $r=\rho_j$, the product equals $1$ when $k=j$ and $0$ otherwise. Thus this polynomial takes the value $P_{M,s}(\rho_j)=\nu_M(Q_{M,s}(\rho_j))$ at each of the $M$ distinct points. Two polynomials of degree at most $M-1$ agreeing at $M$ distinct points are identical, which proves the displayed formula throughout $(\alpha,\beta)$. All nodes and sampled values are rational, so the interpolation polynomial has rational coefficients.

For irrational $r$ in this interval, the right-hand side is exactly the definition of $P_{M,s}(r)$ in~\eqref{eq:gap-interpolation}. For rational $r$ in the interval, Lemma~\ref{lem:finite-difference} identifies the value in~\eqref{eq:rational-coefficient} with the same gap measure. Hence both definitions agree with this single polynomial throughout $(\alpha,\beta)$. At an interval endpoint lying in $I_{M,s}$, the rational formula again gives the gap measure; continuity from Lemma~\ref{lem:polynomial} shows that the adjacent polynomial pieces approach that value. Outside $I_{M,s}$, both $P_{M,s}$ and the gap measure vanish by the definition in the Introduction and Lemma~\ref{lem:gap-support}.
\end{proof}

\begin{proof}[Proof of Theorem~\ref{thm:finite}]
Fix $q=w/\ell\in\Qn$ and, for $1\leq j\leq \lfloor n/\ell\rfloor$, put $M=j\ell$ and $s=jw$. If $s\geq1$ and $r\in I_{M,s}$, Propositions~\ref{prop:moments} and~\ref{prop:gap-evaluation} give
\[
 \begin{aligned}
 B_{q,j}(n,r)
 &=\E_n\binom{N_q(n,r)}j\\
 &=\binom nM(1-Mr+s)^{n-M}\nu_M(\mathcal S_{M,s}(r))\\
 &=\binom nM(1-Mr+s)^{n-M}P_{M,s}(r)
  =W_{M,s}(n,r).
 \end{aligned}
\]
If $s\geq1$ and $r\notin I_{M,s}$, both the moment and the weight are zero. For $q=0$, we have $\ell=1$, $M=j$ and $s=0$, and the explicit part of Proposition~\ref{prop:moments} gives
\[
 B_{0,j}(n,r)=\binom nj(1-jr)_+^{n-1}=W_{j,0}(n,r).
\]
Thus $B_{q,j}(n,r)=W_{j\ell,jw}(n,r)$ for every admissible $q$ and $j\geq1$.

Substituting these moments into Lemma~\ref{lem:inversion} gives the two probability identities
\[
 \begin{aligned}
 \Pp_n(N_q(n,r)\geq1)
 &=\sum_{j=1}^{\lfloor n/\ell\rfloor}(-1)^{j-1}W_{j\ell,jw}(n,r),\\
 \Pp_n(N_q(n,r)=k)
 &=\sum_{j=k}^{\lfloor n/\ell\rfloor}(-1)^{j-k}\binom jk W_{j\ell,jw}(n,r),
 \qquad k\geq1,
 \end{aligned}
\]
where an empty sum is zero. By Proposition~\ref{prop:classifier}, the following event identifications hold up to $\Pp_n$-null sets. The odd-sphere event is the disjoint union of $\{N_q(n,r)\geq1\}$ over $c_t<q<c_{t+1}$. Hence
\[
 \begin{aligned}
 p_{S^{2t+1}}(n,r)
 &=\sum_{\substack{q=w/\ell\in\Qn\\c_t<q<c_{t+1}}}
       \Pp_n(N_q(n,r)\geq1)\\
 &=\sum_{\substack{q=w/\ell\in\Qn\\c_t<q<c_{t+1}}}
       \sum_{j=1}^{\lfloor n/\ell\rfloor}
       (-1)^{j-1}W_{j\ell,jw}(n,r),
 \end{aligned}
\]
which proves~\eqref{eq:odd-law}.

For $m\geq1$, Proposition~\ref{prop:classifier} shows that a wedge of $m$ copies of $S^{2t}$ requires $m+1$ periodic orbits, each containing $2t+1$ points. If $(m+1)(2t+1)>n$, this is impossible, and both sides of~\eqref{eq:even-law} are zero. Otherwise, $c_t=t/(2t+1)\in\Qn$, and the event of this homotopy type is $\{N_{c_t}(n,r)=m+1\}$. Taking $\ell=2t+1$, $w=t$ and $k=m+1$ in the second inversion identity gives
\[
 \begin{aligned}
 p_{\bigvee^m S^{2t}}(n,r)
 &=\Pp_n(N_{c_t}(n,r)=m+1)\\
 &=\sum_{j=m+1}^{\lfloor n/(2t+1)\rfloor}
       (-1)^{j-m-1}\binom{j}{m+1}
       W_{j(2t+1),jt}(n,r),
 \end{aligned}
\]
proving~\eqref{eq:even-law}.

Although $\bigvee^0 S^{2t}\simeq\pt$ for every $t\geq0$, setting $k=1$ at the fixed fraction $c_t$ counts only $\{N_{c_t}(n,r)=1\}$. The contractible event is the disjoint union of these events over all admissible critical fractions. Thus
\[
 \begin{aligned}
 p_{\pt}(n,r)
 &=\sum_{\tau=0}^{\lfloor(n-1)/2\rfloor}
       \Pp_n(N_{c_\tau}(n,r)=1)\\
 &=\sum_{\tau=0}^{\lfloor(n-1)/2\rfloor}
       \sum_{j=1}^{\lfloor n/(2\tau+1)\rfloor}
       (-1)^{j-1}j\,W_{j(2\tau+1),j\tau}(n,r).
 \end{aligned}
\]
Here the factor $j$ is $\binom j1$ from the second inversion identity. This proves~\eqref{eq:point-law} and explains the additional sum over critical fractions. Proposition~\ref{prop:classifier} excludes all other homotopy types.

It remains to prove continuity and the stated polynomial properties on each interval. For $s\geq1$, consider an interval between consecutive elements of $\mathcal R_M$ contained in $I_{M,s}$. Lemma~\ref{lem:polynomial} and Proposition~\ref{prop:gap-evaluation} show that $P_{M,s}$ is a polynomial there with rational coefficients and degree at most $M-1$. Therefore
\[
 \deg_r W_{M,s}(n,r)
 \leq(n-M)+\deg_r P_{M,s}(r)
 \leq(n-M)+(M-1)=n-1.
\]
The same product formula and continuity of the gap measure show that the weight is continuous throughout $I_{M,s}$. Outside $I_{M,s}$ the weight is zero. To check continuity at an endpoint $r_0\in\partial I_{M,s}\cap(0,1/2)$, note that $0<1-Mr+s<1$ on $I_{M,s}$ and $n-M\geq0$. Consequently,
\[
 0\leq W_{M,s}(n,r)
 \leq\binom nM P_{M,s}(r)
 =\binom nM\nu_M(Q_{M,s}(r))
 \qquad(r\in I_{M,s}).
\]
The gap measure is continuous by Lemma~\ref{lem:polynomial}. At $r_0$, Lemma~\ref{lem:gap-support} and Proposition~\ref{prop:gap-evaluation} give
\[
 \nu_M(Q_{M,s}(r_0))=\nu_M(\mathcal S_{M,s}(r_0))=0.
\]
The bound therefore gives
\[
 \lim_{\substack{r\to r_0\\r\in I_{M,s}}}W_{M,s}(n,r)
 =0=W_{M,s}(n,r_0).
\]
Together with the zero value outside $I_{M,s}$, this proves continuity at every endpoint lying in $(0,1/2)$, including when $n=M$.

For $s=0$, the explicit weight is
\[
 W_{M,0}(n,r)
 =\begin{cases}
  \displaystyle\binom nM(1-Mr)^{n-1},&\text{if }r<1/M,\\[3pt]
  0,&\text{if }r\geq1/M.
 \end{cases}
\]
Since $n\geq2$, the two pieces meet at zero and have degree at most $n-1$. Any breakpoint $1/M$ in $(0,1/2)$ belongs to $\mathcal R_M$. For $M=1$, the weight $W_{1,0}(n,r)=n(1-r)^{n-1}$ is a polynomial on the whole scale interval.

Since $\mathcal R_M\subseteq\mathcal R_n$ for $2\leq M\leq n$, the finite sums defining $p_T$ inherit continuity and the polynomial properties of the weights on the common partition $\mathcal R_n$.
\end{proof}

\begin{remark}\label{rem:low-scale}
For $n\geq2$ and $0<r\leq1/3$, the following known circle law is a special case of Theorem~\ref{thm:finite}:
\[
 p_{S^1}(n,r)=1-\sum_{j=1}^n(-1)^{j-1}\binom nj(1-jr)_+^{n-1}.
\]
Indeed, in this range Proposition~\ref{prop:classifier} identifies the circle event with $\{N_0(n,r)=0\}$ almost surely. Binomial inversion with the zero-winding weights $W_{j,0}$ in~\eqref{eq:weights} gives the displayed formula. The count $N_0(n,r)$ is the number of circular gaps of length at least $r$, whose classical distribution is due to Stevens~\cite[Sections~3--5]{Stevens}. For $0<r<1/3$, the exact circle probability also appears in~\cite[Section~5]{PVK}; the endpoint $r=1/3$ follows by continuity.
\end{remark}

\section{The critical Poisson and homotopy laws}\label{sec:critical}
Fix $t\in\Z_{\geq1}$ and put $q=c_t=t/(2t+1)$, $w=t$ and $\ell=2t+1$. In this section $c_t$ is the prescribed target fraction; the actual sample winding fraction need not equal it. We first specialize the lattice-count evaluation to one critical orbit, then estimate the joint persistence of several such orbits to obtain their moment asymptotics.

\subsection{The critical one-orbit probability}\label{sec:critical-one-orbit}
Selecting one prospective critical orbit means taking $j=1$, hence $M=\ell$ and $s=t$. The probability $P_{\ell,t}$ concerns the map on those $\ell$ points alone. Since $\gcd(t,\ell)=1$, the shift $i\mapsto i+t$ consists of one orbit. The near-regular configurations and threshold analysis have antecedents in~\cite[Section~6]{AA}.

\begin{lemma}\label{lem:critical-probability}
For $0<r<1/2$,
\begin{equation}\label{eq:critical-probability}
 P_{\ell,t}(r)=(\ell r-t)_+^{\ell-1}.
\end{equation}
\end{lemma}
\begin{proof}
Here $I_{\ell,t}=(c_t,1/2)$, so $P_{\ell,t}(r)=0$ for $r\leq c_t$ by the definition following~\eqref{eq:support}. First take a rational scale $r=A/B\in(c_t,1/2)$ in lowest terms and put $D=\ell A-tB>0$.

Fix an integer $k\geq1$ and a vector $d\in\Z_{\geq0}^{\ell}$ counted by $C_{\ell,t}(kB,kA)$. All subscripts below are interpreted modulo $\ell$. Define
\begin{equation}\label{eq:critical-integer-map}
 u_i=kA-\sum_{v=0}^{t-1}d_{i+v},
 \qquad i=0,\ldots,\ell-1.
\end{equation}
Under~\eqref{eq:critical-integer-map}, the two inequalities defining $C_{\ell,t}(kB,kA)$ become
\[
 0\leq u_i\leq d_{i+t},\qquad i=0,\ldots,\ell-1.
\]
In particular, every $u_i$ is nonnegative. Each $d_i$ occurs in exactly $t$ of these block sums, so
\[
 \sum_{i=0}^{\ell-1}u_i
 =\ell kA-t\sum_{i=0}^{\ell-1}d_i=kD.
\]
The converse argument below shows that the upper bounds $u_i\leq d_{i+t}$ hold automatically for the reconstructed $d$, so they impose no additional restriction on nonnegative $u$ with total $kD$.

The two $t$-blocks defining $u_{i-t}$ and $u_{i+1}$ contain every $d$-coordinate except $d_i$. Hence $u_{i-t}+u_{i+1}=2kA-(kB-d_i)$, giving the inverse formula
\begin{equation}\label{eq:critical-integer-inverse}
 d_i=k(B-2A)+u_{i-t}+u_{i+1}.
\end{equation}
The complementary-block identity underlying this inverse also appears in~\cite[Lemma~6.3(b), proof]{AA}.

Conversely, take any $u\in\Z_{\geq0}^{\ell}$ with $\sum_i u_i=kD$ and define $d$ by~\eqref{eq:critical-integer-inverse}. Since $B>2A$, every $d_i$ is nonnegative, and
\[
 \begin{aligned}
 \sum_{i=0}^{\ell-1}d_i
 &=\ell k(B-2A)+2kD=kB,\\
 \sum_{v=0}^{t-1}d_{i+v}
 &=tk(B-2A)+kD-u_i=kA-u_i\leq kA.
 \end{aligned}
\]
In the second line, the two $t$-blocks of $u$ contain every coordinate except $u_i$. This equality also shows that~\eqref{eq:critical-integer-map} recovers the prescribed $u$. The remaining inequalities follow by complementation:
\[
 \sum_{v=0}^{t}d_{i+v}
 =kB-\sum_{v=0}^{t-1}d_{i+t+1+v}
 \geq k(B-A)>kA.
\]
Thus the two formulas are mutually inverse integer maps between the vectors counted by $C_{\ell,t}(kB,kA)$ and the nonnegative integer solutions of $\sum_i u_i=kD$. Counting the latter by placing $\ell-1$ separators among $kD+\ell-1$ positions gives
\begin{equation}\label{eq:critical-integer-count}
 C_{\ell,t}(kB,kA)
 =\#\left\{u\in\Z_{\geq0}^{\ell}:\sum_{i=0}^{\ell-1}u_i=kD\right\}
 =\binom{kD+\ell-1}{\ell-1}.
\end{equation}
This also holds for $k=0$, when both sides count only the zero vector.

As a polynomial in $k$,
\[
 \binom{kD+\ell-1}{\ell-1}
 =\frac{(kD+1)\cdots(kD+\ell-1)}{(\ell-1)!}
\]
has degree $\ell-1$ and leading coefficient $D^{\ell-1}/(\ell-1)!$. By the leading-coefficient calculation in the proof of Lemma~\ref{lem:finite-difference}, its $(\ell-1)$st forward difference is $D^{\ell-1}$. Substituting~\eqref{eq:critical-integer-count} into the definition~\eqref{eq:rational-coefficient} therefore gives
\[
 \begin{aligned}
 P_{\ell,t}(A/B)
 &=\frac{1}{B^{\ell-1}}\sum_{k=0}^{\ell-1}
    (-1)^{\ell-1-k}\binom{\ell-1}{k}
    \binom{kD+\ell-1}{\ell-1}\\
 &=\frac{D^{\ell-1}}{B^{\ell-1}}
  =(\ell A/B-t)^{\ell-1}.
 \end{aligned}
\]
Finally, $c_t=t/\ell$ and $1/2$ both belong to $\mathcal R_\ell$. Thus every interval between consecutive elements of $\mathcal R_\ell$ that meets $(c_t,1/2)$ lies entirely in this range, and its $\ell$ distinct rational interpolation nodes lie there as well. Since $r\mapsto(\ell r-t)^{\ell-1}$ is a polynomial of degree $\ell-1$, the Lagrange formula~\eqref{eq:gap-interpolation} reproduces it from the rational values just computed. The partition points in $(c_t,1/2)$ are themselves rational and were already covered above. This proves the equality for every real $r\in(c_t,1/2)$; together with the zero value for $r\leq c_t$, it gives~\eqref{eq:critical-probability}. By Proposition~\ref{prop:gap-evaluation}, this coefficient equals the gap probability $\nu_\ell(\mathcal S_{\ell,t}(r))$ throughout $0<r<1/2$.
\end{proof}

\subsection{Candidate orbits and simultaneous persistence}
We now estimate the probability that several candidate periodic orbits persist simultaneously when all sample points are included.

Fix $0<\delta<1/2-c_t$ and $1\leq j\leq\lfloor n/\ell\rfloor$. Let $\mathfrak G_{n,j}$ be the set of all unordered collections of $j$ pairwise disjoint $\ell$-element subsets of $[n]=\{1,\ldots,n\}$. Its cardinality is
\begin{equation}\label{eq:collections}
 D_{n,j}=\binom n{j\ell}\frac{(j\ell)!}{j!(\ell!)^j}
        =\frac{n!}{j!(\ell!)^j(n-j\ell)!}.
\end{equation}
For a fixed index set $G\subseteq[n]$ of size $\ell$, specialize the event $E_L$ of Section~\ref{subsec:preserving-orbits} to $L=G$, $M=\ell$, $s=t$ and $r=c_t+\delta$:
\[
 E_G=\{\omega\in\Omega_n:
       H_\ell(\pr_G(\omega))\in\mathcal S_{\ell,t}(c_t+\delta)\}.
\]
By the fixed-index marginal law following~\eqref{eq:projection}, $H_\ell\circ\pr_G$ has law $\nu_\ell$. Hence Lemma~\ref{lem:critical-probability} and $c_t=t/\ell$ yield
\[
 \begin{aligned}
 \Pp_n(E_G)
 &=\nu_\ell\bigl(\mathcal S_{\ell,t}(c_t+\delta)\bigr)
   =P_{\ell,t}(c_t+\delta)\\
 &=\bigl(\ell(c_t+\delta)-t\bigr)^{\ell-1}
   =(\ell\delta)^{\ell-1}.
 \end{aligned}
\]
For $\mathcal G\in\mathfrak G_{n,j}$, set
\[
 E_{\mathcal G}=\bigcap_{G\in\mathcal G}E_G.
\]
Since the sets $G$ in $\mathcal G$ are pairwise disjoint, these events depend on disjoint blocks of independent coordinates. Therefore
\begin{equation}\label{eq:collection-probability}
 \Pp_n(E_{\mathcal G})
 =\prod_{G\in\mathcal G}\Pp_n(E_G)
 =(\ell\delta)^{j(\ell-1)}.
\end{equation}
We next count the collections for which $E_{\mathcal G}$ occurs. By linearity of expectation and~\eqref{eq:collections}, their expected number is
\begin{equation}\label{eq:candidate-collection-mean}
 \begin{aligned}
 A_{n,j}(\delta)
 &:=\E_n\sum_{\mathcal G\in\mathfrak G_{n,j}}
       \ind_{E_{\mathcal G}}
   =\sum_{\mathcal G\in\mathfrak G_{n,j}}
       \Pp_n(E_{\mathcal G})\\
 &=D_{n,j}(\ell\delta)^{j(\ell-1)}.
 \end{aligned}
\end{equation}
Linearity of expectation applies even when different collections share indices. Thus $A_{n,j}(\delta)$ is the expected number of collections satisfying the successor condition within each group. Persistence of these groups in the full sample is addressed next.

\begin{lemma}\label{lem:persistence}
For the parameters above,
\begin{equation}\label{eq:persistence-bound}
 A_{n,j}(\delta)(1-\varepsilon_{n,j}(\delta))_+
 \leq B_{c_t,j}(n,c_t+\delta)\leq A_{n,j}(\delta),
\end{equation}
where
\begin{equation}\label{eq:error-bound}
 \varepsilon_{n,j}(\delta)
   =j(j-1)\ell^2\delta+(n-j\ell)j\ell\delta.
\end{equation}
\end{lemma}
\begin{proof}
Fix $\mathcal G=\{G_1,\ldots,G_j\}$, using an arbitrary enumeration solely for the proof. For each $b$, let $V_b$ be the forbidden union~\eqref{eq:avoidance-arcs} associated with $G_b$. We first bound, conditional on $E_{\mathcal G}$, the probability that a point indexed by another selected group $G_{b'}$, $b'\ne b$, enters $V_b$ and changes a prescribed successor of $G_b$. We then treat the unselected points.

Since~\eqref{eq:collection-probability} is positive, the conditional probability measure
\[
 \Pp_n^{\mathcal G}(E)
 =\frac{\Pp_n(E\cap E_{\mathcal G})}
        {\Pp_n(E_{\mathcal G})},\qquad E\in\mathcal F_n,
\]
is well defined. In each $G_b$, anchor at its smallest sample index and write $R_b$ for the anchor position and $\mathbf Y_b$ for the other relative positions, retained in sample-index order. This is the map~\eqref{eq:anchor-map} applied to the coordinates indexed by $G_b$. The vector $\mathbf Y_b$ records the relative configuration of the group, while $R_b$ specifies its rotation around the circle.

Under $\Pp_n$, the anchor map gives a uniform $R_b$ independent of $\mathbf Y_b$. Since different groups use disjoint sample indices, the pairs $(R_b,\mathbf Y_b)$ are independent. Hence the anchor vector $(R_1,\ldots,R_j)$ has law $\mu^{\otimes j}$ and is independent of the $\sigma$-algebra
\[
 \mathcal Y_{\mathcal G}=\sigma(\mathbf Y_1,\ldots,\mathbf Y_j).
\]
Each event $E_{G_b}$ is $\sigma(\mathbf Y_b)$-measurable, so $E_{\mathcal G}\in\mathcal Y_{\mathcal G}$. Under $\Pp_n^{\mathcal G}$, the anchor vector therefore still has law $\mu^{\otimes j}$ and is independent of $\mathcal Y_{\mathcal G}$. In particular, for any Borel sets $U_1,\ldots,U_j\subseteq S^1$,
\[
 \begin{aligned}
 &\Pp_n^{\mathcal G}\bigl(R_1\in U_1,\ldots,R_j\in U_j
                      \mid\mathcal Y_{\mathcal G}\bigr)\\
 &\qquad=\Pp_n\bigl(R_1\in U_1,\ldots,R_j\in U_j
                      \mid\mathcal Y_{\mathcal G}\bigr)
       =\prod_{b=1}^j\mu(U_b)
 \end{aligned}
\]
$\Pp_n^{\mathcal G}$-almost surely.

For $b\ne b'$, put $\mathcal K_b=\mathcal Y_{\mathcal G}\vee\sigma(R_b)$. The preceding independence implies that $R_{b'}$ is independent of $\mathcal K_b$ under $\Pp_n^{\mathcal G}$ and has law $\mu$. The endpoints of the arcs defining $V_b$ are $\mathcal K_b$-measurable, and $\mu(V_b)=\ell\delta$ by Lemma~\ref{lem:avoidance}. For $i\in G_{b'}$, let $\tau_i(\mathbf Y_{b'})$ denote its position relative to the anchor: it is zero for the anchor and the corresponding component of $\mathbf Y_{b'}$ otherwise. Then $X_i=R_{b'}+\tau_i(\mathbf Y_{b'})$ modulo one, and $\tau_i(\mathbf Y_{b'})$ is $\mathcal K_b$-measurable. Writing $\E_n^{\mathcal G}$ for expectation under $\Pp_n^{\mathcal G}$, we obtain
\[
 \begin{aligned}
 \Pp_n^{\mathcal G}(X_i\in V_b\mid\mathcal K_b)
 &=\E_n^{\mathcal G}\bigl(
     \ind_{V_b}(R_{b'}+\tau_i(\mathbf Y_{b'}))
       \mid\mathcal K_b\bigr)\\
 &=\int_{S^1}\ind_{V_b}\bigl(u+\tau_i(\mathbf Y_{b'})\bigr)\,d\mu(u)\\
 &=\int_{S^1}\ind_{V_b-\tau_i(\mathbf Y_{b'})}(u)\,d\mu(u)\\
 &=\mu\bigl(V_b-\tau_i(\mathbf Y_{b'})\bigr)
   =\mu(V_b)=\ell\delta.
 \end{aligned}
\]
Here the integral of a set indicator is its measure, and translation invariance of $\mu$ gives the equality of the two set measures. All these identities hold $\Pp_n^{\mathcal G}$-almost surely. Taking expectations yields $\Pp_n^{\mathcal G}(X_i\in V_b)=\ell\delta$.

Almost surely, a point of $G_{b'}$ can change a prescribed successor of $G_b$ only if it lies in $V_b$. A union bound over the $\ell$ points of $G_{b'}$ and the $j(j-1)$ ordered pairs $(b,b')$ therefore gives
\[
 \begin{aligned}
 &\Pp_n^{\mathcal G}\bigl(
    \exists\,b,b'\in\{1,\ldots,j\},\ b\ne b',\
    \exists\,i\in G_{b'}:\ X_i\in V_b\bigr)\\
 &\qquad\leq
 \sum_{\substack{1\leq b,b'\leq j\\b\ne b'}}
 \sum_{i\in G_{b'}}\Pp_n^{\mathcal G}(X_i\in V_b)
 =j(j-1)\ell^2\delta.
 \end{aligned}
\]

Put $L_{\mathcal G}=\bigcup_{b=1}^jG_b$. Under $\Pp_n^{\mathcal G}$, conditional on $\mathcal A_{L_{\mathcal G}}$, the $n-j\ell$ unselected points are independent and uniform. Thus, for every $i\in[n]\setminus L_{\mathcal G}$,
\[
 \Pp_n^{\mathcal G}\left(
 X_i\in\bigcup_{b=1}^jV_b\,\middle|\,
 \mathcal A_{L_{\mathcal G}}\right)
 =\mu\left(\bigcup_{b=1}^jV_b\right)
 \leq j\ell\delta
\]
$\Pp_n^{\mathcal G}$-almost surely. Taking expectations and applying a union bound gives
\[
 \begin{aligned}
 &\Pp_n^{\mathcal G}\bigl(
   \exists\,i\in[n]\setminus L_{\mathcal G}:\
       X_i\in\textstyle\bigcup_{b=1}^jV_b\bigr)\\
 &\qquad\leq
 \sum_{i\in[n]\setminus L_{\mathcal G}}
 \Pp_n^{\mathcal G}\bigl(
       X_i\in\textstyle\bigcup_{b=1}^jV_b\bigr)
 \leq(n-j\ell)j\ell\delta.
 \end{aligned}
\]
Together with the bound for the other selected groups, this bounds the combined failure probability by~\eqref{eq:error-bound}.

Let $F_{\mathcal G}$ be the event that the prescribed groups are $j$ distinct full-sample $c_t$-orbits. On $E_{\mathcal G}$, absence of the failures just counted implies $F_{\mathcal G}$. Conversely, $F_{\mathcal G}$ implies $E_{\mathcal G}$. Indeed, after other points are removed, each retained full-sample successor is still eligible, and no farther clockwise point was eligible even in the full sample. Each prescribed orbit therefore retains its successors. Thus
\[
 \Pp_n(E_{\mathcal G})(1-\varepsilon_{n,j}(\delta))_+
 \leq\Pp_n(F_{\mathcal G})\leq\Pp_n(E_{\mathcal G}).
\]
For almost every sample,
\[
 \sum_{\mathcal G\in\mathfrak G_{n,j}}\ind_{F_{\mathcal G}}
           =\binom{N_{c_t}(n,c_t+\delta)}j.
\]
Taking expectations proves~\eqref{eq:persistence-bound}. The positive part makes the lower bound valid even when the union bound exceeds one.
\end{proof}

\subsection{Poisson convergence}
Recall from~\eqref{eq:critical-parameters} that $\lambda=\ell^{\ell-1}a^{\ell-1}/\ell!$, where $a>0$ is fixed. For the sequence in Theorem~\ref{thm:critical}, put $\delta_n=a n^{-\ell/(\ell-1)}$ and $r_n=c_t+\delta_n$. Then
\begin{equation}\label{eq:scale-limits}
 \lim_{n\to\infty}n\delta_n=0,
 \qquad n^\ell\delta_n^{\ell-1}=a^{\ell-1}.
\end{equation}

\begin{proposition}\label{prop:poisson}
For each integer $k\geq0$,
\begin{equation}\label{eq:poisson-masses}
 \lim_{n\to\infty}\Pp_n(N_{c_t}(n,r_n)=k)
       =e^{-\lambda}\frac{\lambda^k}{k!}.
\end{equation}
In particular, the laws of $N_{c_t}(n,r_n)$ converge weakly to $\Pois(\lambda)$ as $n\to\infty$.
\end{proposition}
\begin{proof}
Fix $j\geq1$. For sufficiently large $n$, $j\ell\leq n$ and
\[
 A_{n,j}(\delta_n)
 =\frac{\lambda^j}{j!}
    \frac{n!}{(n-j\ell)!\,n^{j\ell}},\qquad
 \lim_{n\to\infty}\varepsilon_{n,j}(\delta_n)=0.
\]
Lemma~\ref{lem:persistence} therefore yields
\begin{equation}\label{eq:critical-moment-limit}
 \lim_{n\to\infty}B_{c_t,j}(n,r_n)=\frac{\lambda^j}{j!}.
\end{equation}
We now use the binomial-moment formula~\eqref{eq:binomial-moments} for every integer $j\geq0$. Each $c_t$-orbit contains $\ell$ points, and distinct periodic orbits are disjoint. Thus $N_{c_t}(n,r_n)\leq\lfloor n/\ell\rfloor$ almost surely, and consequently
\[
 B_{c_t,j}(n,r_n)
 =\E_n\binom{N_{c_t}(n,r_n)}j=0
 \qquad\bigl(j>\lfloor n/\ell\rfloor\bigr).
\]
Also, $B_{c_t,0}(n,r_n)=1$. The upper bound in Lemma~\ref{lem:persistence}, together with these values, gives, for all sufficiently large $n$ and every integer $j\geq0$,
\begin{equation}\label{eq:domination}
 0\leq B_{c_t,j}(n,r_n)\leq\frac{\lambda^j}{j!},
\end{equation}
since $n!/(n-j\ell)!\leq n^{j\ell}$ whenever $j\ell\leq n$.

For fixed $k\geq0$ and all sufficiently large $n$, Lemma~\ref{lem:inversion} gives
\[
 \begin{aligned}
 \Pp_n(N_{c_t}(n,r_n)=k)
 &=\sum_{j=k}^{\lfloor n/\ell\rfloor}
       (-1)^{j-k}\binom jk B_{c_t,j}(n,r_n)\\
 &=\sum_{j=k}^{\infty}(-1)^{j-k}\binom jk B_{c_t,j}(n,r_n).
 \end{aligned}
\]
The added terms vanish, so the second sum has the same value as the first. Its index set $\{k,k+1,\ldots\}$ is independent of $n$. To apply dominated convergence on this fixed index set, we use~\eqref{eq:domination} and the summable dominating sequence $\binom jk\lambda^j/j!$, for which
\[
 \sum_{j=k}^{\infty}\binom jk\frac{\lambda^j}{j!}
       =\frac{\lambda^k}{k!}e^\lambda<\infty.
\]
Dominated convergence for the counting measure therefore permits the interchange of the limit and the sum. Using~\eqref{eq:critical-moment-limit}, we obtain
\[
 \begin{aligned}
 \lim_{n\to\infty}\Pp_n(N_{c_t}(n,r_n)=k)
 &=\sum_{j=k}^{\infty}(-1)^{j-k}\binom jk
         \lim_{n\to\infty}B_{c_t,j}(n,r_n)\\
 &=\sum_{j=k}^{\infty}(-1)^{j-k}\binom jk\frac{\lambda^j}{j!}
 =e^{-\lambda}\frac{\lambda^k}{k!}.
 \end{aligned}
\]

These limiting masses sum to one. Their convergence therefore gives tightness and weak convergence on $\Z_{\geq0}$, even though the sampling spaces vary with $n$.
\end{proof}
This is a direct factorial-moment argument for the present full-sample orbit events. General rare-tuple Poisson results, including the binomial setting, are given in~\cite[Theorem~2 and Corollary~2]{ST}; see also~\cite{SB}. The calculation above supplies the specific intensity and persistence estimate needed here.

\subsection{Winding localization and the limiting homotopy law}
Let $q_n(\omega)$ be the common winding fraction of the map $f_{r_n}$ associated with $\Xn(\omega)$, with arbitrary values on the collision null set. We distinguish this random variable from the fixed target $c_t$.

\begin{lemma}\label{lem:localization}
Fix $t\in\Z_{\geq1}$ and $a>0$. For all sufficiently large $n$,
\begin{equation}\label{eq:localization}
 \Pp_n(c_{t-1}<q_n\leq c_t)\geq1-\frac1n.
\end{equation}
\end{lemma}
\begin{proof}
Write the realized fraction as $q_n=w'/\ell'$ in lowest terms, with $\ell'\leq n$. If $q_n>c_t=t/\ell$, then $\ell w'-t\ell'$ is a positive integer and is therefore at least one. Hence
\[
 q_n-c_t=\frac{\ell w'-t\ell'}{\ell\ell'}
                   \geq\frac1{\ell\ell'}
                   \geq\frac1{\ell n}.
\]
Every orbit step is shorter than $r_n$, so~\eqref{eq:period-winding} gives $q_n<r_n$. Since $n\delta_n\to0$ as $n\to\infty$, for large $n$ the inequalities $q_n>c_t$ would force
\[
 c_t+\frac1{\ell n}\leq q_n<r_n=c_t+\delta_n
                           <c_t+\frac1{\ell n},
\]
a contradiction. Thus $q_n\leq c_t$ almost surely for these $n$.

For the lower bound, apply~\cite[Lemma~5.2]{AAM} at the deterministic scale $r_n$. For all sufficiently large $n$, it gives
\begin{equation}\label{eq:aam-bound}
 \Pp_n\left(q_n\geq r_n-\frac{2\log n}{n}\right)\geq1-\frac1n.
\end{equation}
Moreover,
\[
 \lim_{n\to\infty}\left(r_n-\frac{2\log n}{n}\right)=c_t,
 \qquad c_t-c_{t-1}=\frac1{(2t+1)(2t-1)}>0.
\]
For large $n$, the convergence above gives $r_n-2\log n/n>(c_t+c_{t-1})/2>c_{t-1}$. Thus the event in~\eqref{eq:aam-bound} implies $q_n>c_{t-1}$. Since $\Pp_n(q_n\leq c_t)=1$, we obtain
\[
 \Pp_n(c_{t-1}<q_n\leq c_t)
 =\Pp_n(q_n>c_{t-1})
 \geq\Pp_n\left(q_n\geq r_n-\frac{2\log n}{n}\right)
 \geq1-\frac1n.\qedhere
\]
\end{proof}

\begin{proof}[Proof of Theorem~\ref{thm:critical}]
Let $\mathcal E_n=\{\omega\in\Omega_n:c_{t-1}<q_n(\omega)\leq c_t\}$. Lemma~\ref{lem:localization} gives $\Pp_n(\mathcal E_n^c)\leq1/n$ for all sufficiently large $n$. By the definition of $N_{c_t}$, we have $N_{c_t}(n,r_n)\geq1$ if and only if $q_n=c_t$. Combining this observation with Proposition~\ref{prop:classifier} gives the following dictionary on $\mathcal E_n$:
\begin{equation}\label{eq:critical-dictionary}
 \renewcommand{\arraystretch}{1.2}
 \begin{array}{|c|c|c|}\hline
 N_{c_t}(n,r_n)&\text{winding fraction }q_n&\text{homotopy type of }\VR(\Xn;r_n)\\ \hline\hline
 0&c_{t-1}<q_n<c_t&S^{2t-1}\\ \hline
 1&q_n=c_t&\pt\\ \hline
 m+1\ (m\geq1)&q_n=c_t&\displaystyle\bigvee^m S^{2t}\\ \hline
 \end{array}
\end{equation}
All event identifications below are understood up to the fixed-scale strict/closed null event.

For the odd sphere, Proposition~\ref{prop:classifier} identifies its event with $\{c_{t-1}<q_n<c_t\}=\{N_{c_t}(n,r_n)=0\}\cap\mathcal E_n$. Consequently,
\[
 \begin{aligned}
 0&\leq\Pp_n(N_{c_t}(n,r_n)=0)-p_{S^{2t-1}}(n,r_n)\\
  &=\Pp_n\bigl(\{N_{c_t}(n,r_n)=0\}\cap\mathcal E_n^c\bigr)
    \leq\Pp_n(\mathcal E_n^c)\leq\frac1n.
 \end{aligned}
\]
Applying Proposition~\ref{prop:poisson} with $k=0$ therefore gives
\[
 \lim_{n\to\infty}p_{S^{2t-1}}(n,r_n)
 =\lim_{n\to\infty}\Pp_n(N_{c_t}(n,r_n)=0)
 =e^{-\lambda}\frac{\lambda^0}{0!}=e^{-\lambda}.
\]

We treat the last two rows of~\eqref{eq:critical-dictionary} together. Fix $k\geq1$ and interpret $\bigvee^0 S^{2t}$ as $\pt$. On $\mathcal E_n$, the type $\bigvee^{k-1}S^{2t}$ occurs if and only if $N_{c_t}(n,r_n)=k$. The latter event is contained in $\mathcal E_n$, since a positive count forces $q_n=c_t$. Therefore
\[
 \begin{aligned}
 0&\leq p_{\bigvee^{k-1}S^{2t}}(n,r_n)
          -\Pp_n(N_{c_t}(n,r_n)=k)\\
  &=\Pp_n\bigl(\{\VR(\Xn;r_n)\simeq\bigvee^{k-1}S^{2t}\}
                 \cap\mathcal E_n^c\bigr)
    \leq\Pp_n(\mathcal E_n^c)\leq\frac1n.
 \end{aligned}
\]
Proposition~\ref{prop:poisson} now gives the single calculation
\[
 \lim_{n\to\infty}p_{\bigvee^{k-1}S^{2t}}(n,r_n)
 =\lim_{n\to\infty}\Pp_n(N_{c_t}(n,r_n)=k)
 =e^{-\lambda}\frac{\lambda^k}{k!}.
\]
Taking $k=1$ and $k=m+1$ gives~\eqref{eq:critical-point} and~\eqref{eq:critical-even}, respectively. Finally, the table exhausts the possibilities on $\mathcal E_n$, so the total probability of all remaining homotopy types satisfies
\[
 \begin{aligned}
 0&\leq1-p_{S^{2t-1}}(n,r_n)-p_{\pt}(n,r_n)
          -\sum_{m=1}^{\infty}p_{\bigvee^m S^{2t}}(n,r_n)\\
  &\leq\Pp_n(\mathcal E_n^c)\leq\frac1n
          \xrightarrow[n\to\infty]{}0.
 \end{aligned}
\]
This proves the last assertion.
\end{proof}

\end{document}